\documentclass[11pt,oneside,english,reqno]{amsart}
\usepackage[T1]{fontenc}
\usepackage[latin9]{inputenc}
\usepackage{amsthm}
\usepackage{amstext}
\usepackage{amssymb}
\usepackage{fancyhdr}
\makeatletter

\theoremstyle{plain}
\newtheorem{thm}{\protect\theoremname}[section]
  \theoremstyle{plain}
  \newtheorem{lem}[thm]{\protect\lemmaname}
  \theoremstyle{remark}
  \newtheorem{claim}[thm]{\protect\claimname}
  \theoremstyle{plain}
  \newtheorem{prop}[thm]{\protect\propositionname}
  \theoremstyle{remark}
  \newtheorem{rem}[thm]{\protect\remarkname}
  \theoremstyle{definition}
  \newtheorem{defn}[thm]{\protect\definitionname}
  \theoremstyle{plain}
  \newtheorem{cor}[thm]{\protect\corollaryname}
  \newcounter{casectr}

 \newcommand{\om}{\omega}
 \newcommand{\sig}{\sigma}   
 
 \newcommand{\del}{\delta}  \newcommand{\Del}{\Delta}
    
 \newcommand{\lam}{\lambda}   
 \newcommand{\eps}{\varepsilon}
  \newcommand{\wed}{\wedge}
 \newcommand{\ip}[1]{\langle #1 \rangle}
 \renewcommand{\tilde}{\widetilde}

 \def\b1{\text{\bf\large 1}}  %big bold 1

 \def\simto{\overset{\sim}{\longrightarrow}}
 
 \def\ip<#1>{\langle#1\rangle}   %inner product--write `\ip<contents>'

 \newcommand{\Hom}{\operatorname{Hom}}

 \newcommand{\Lie}{\operatorname{Lie}}
 \newcommand{\lt}{\operatorname{lt.}}
 \newcommand{\Vir}{\operatorname{Vir}}

 \newcommand{\fg}{\mathfrak{g}}

\newcommand{\fl}{\mathfrak{l}}

 \newcommand{\fu}{\mathfrak{u}}

\newcommand{\Invlt}{\operatorname{Invlt}}

\newcommand{\bc}{\mathbb{C}}

\newcommand{\br}{\mathbb{R}}
\newcommand{\bz}{\mathbb{Z}}

 \newcommand{\ca}{\mathcal{A}}

 \newcommand{\cf}{\mathcal{F}}
 
 \newcommand{\cl}{\mathcal{L}}
 
 \newcommand{\co}{\mathcal{O}}
 
 \newcommand{\cs}{\mathcal{S}}

\makeatother

\usepackage{babel}
  \providecommand{\casename}{Case}
  \providecommand{\claimname}{Claim}
  \providecommand{\corollaryname}{Corollary}
  \providecommand{\definitionname}{Definition}
  \providecommand{\lemmaname}{Lemma}
  \providecommand{\propositionname}{Proposition}
  \providecommand{\remarkname}{Remark}
  
\providecommand{\theoremname}{Theorem}

\newcommand{\beqn}{\begin{equation}}
\newcommand{\eeqn}{\end{equation}}

\swapnumbers
\theoremstyle{plain}
\newtheorem{theorem}[thm]{Theorem}
\newtheorem{lemma}[thm]{Lemma}
\newtheorem{proposition}[thm]{Proposition}
\newtheorem{corollary}[thm]{Corollary}

\newtheorem{conjecture}[thm]{Conjecture}
\theoremstyle{definition}
\newtheorem{remark}[thm]{Remark}
\newtheorem{definition}[thm]{Definition}

\global\long\def\cl{\mathcal{L}}
\global\long\def\lam{\lambda}
\global\long\def\bc{\mathbb{C}}

\global\long\def\eps{\varepsilon}
\global\long\def\sig{\sigma}
\global\long\def\del{\delta}
\global\long\def\lt{\operatorname{lim}}
\global\long\def\ip#1{\langle#1 \rangle}
\global\long\def\Del{\Delta}
\global\long\def\co{\mathcal{O}}
\global\long\def\cs{\mathcal{S}}
\global\long\def\fh{\mathfrak{h}}
\global\long\def\fg{\mathfrak{g}}
\newcommand\exx{\Bbb{X}}
 \newcommand\elal{\Bbb{L}}

\newcommand{\optop}{\operatorname{top}}
\newcommand{\tor}{\mathcal{T}or}
\newcommand{\ext}{\mathcal{E}xt}
\newcommand{\SL}{\operatorname{SL}}
\def\u.{^{\bullet}}
\newcommand{\gr}{\operatorname{gr}}
\begin{document}

  \title[Saturated tensor cone for symmetrizable
  Kac-Moody algebras]{A study of saturated tensor cone for symmetrizable
  Kac-Moody algebras}
  \author{Merrick Brown and Shrawan Kumar }

\maketitle

\section{Introduction}\label{intro}

Let $\fg$ be a symmetrizable  Kac-Moody Lie algebra with  the standard  Cartan subalgebra $\fh$ and the Weyl group $W$. Let $P_+$ be the set of dominant integral weights. For $\lambda \in P_+$, let $L(\lambda)$ be the irreducible, integrable, highest weight representation of $\fg$ with highest weight $\lambda$. For a positive integer $s$, define the {\em saturated tensor semigroup} as
\begin{align*}
\Gamma_s:= \{(\lambda_1, \dots, \lambda_s,\mu)\in P_+^{s+1}: \exists\,
N>1 \,\,\text{with}\,\, L(N\mu)\subset L(N\lambda_1)\otimes \dots \otimes L(N\lambda_s)\}.
\end{align*}
The aim of this paper is to begin  a systematic study of $\Gamma_s$ in the infinite dimensional symmetrizable Kac-Moody case. 
In this paper, we produce a set of necessary inequalities satisfied by $\Gamma_s$, which we describe now. Let $X =G^{\min}/B$  be the standard full KM-flag variety
associated to $\fg$, where $G^{\min}$ is the `minimal' Kac-Moody group with Lie algebra $\fg$ and $B$ is the standard Borel subgroup of 
$G^{\min}$.  For $w\in W$, let  $X_w = \overline{BwB/B}\subset X$ be the corresponding Schubert variety. Let $\{\eps^w\}_{w\in W} \subset H^*(X,\bz)$ be the 
(Schubert) basis dual (with respect to  the standard pairing) to the basis of the singular homology of $X$ given by the fundamental classes of $X_w$.  
The following result is our first main theorem valid for any symmetrizable $\fg$ (cf. Theorem \ref{thm1}).
\begin{thm}
Let
 $(\lambda_1,\dots,\lambda_s,\mu) \in \Gamma_s$. Then, 
for any $u_1, \dots , u_s, v\in W$ such that $n^v_{u_1, \dots, u_s}\neq 0$, where
$$\eps^{u_1} \dots \eps^{u_s} = \sum_w n^w_{u_1, \dots, u_s}\,\eps^w,$$
we have 
\[\left(\sum^s_{j=1} \lam_j (u_jx_i)\right) - \mu (vx_i) \geq 0, \,\,\,\text{for any}\,\, x_i,\]
where 
 $x_i\in \fh$ is dual to the simple roots of $\fg$.
\end{thm}
The proof of the theorem relies on the Kac-Moody analogue of the Borel-Weil theorem and the Geometric Invariant Theory (specifically the Hilbert-Mumford index).  We conjecture that the above inequalities are sufficient as well to describe $\Gamma_s$. In fact, we conjecture a much sharper result,
where much fewer inequalities suffice to describe the semigroup $\Gamma_s$. To explain our conjecture, we need some more notation. 

Let $P\supset B$ be a (standard) parabolic subgroup and let $X_P:=G^{\min}/P$ be the corresponding partial flag variety. Let $W_P$ be the Weyl group of $P$ (which is, by definition, the Weyl group of the Levi $L$ of $P$) and let $W^P$ be the set of minimal length coset representatives 
of cosets in $W/W_P$. 
The projection map $X\to X_P$ induces an injective homomorphism $H^*(X_P, \bz) \to H^*(X, \bz)$ and $H^*(X_P, \bz) $ has the Schubert basis
$\{\eps^w_P\}_{w\in W^P}$ such that $\eps^w_P$ goes to   $\eps^w$ for any $w\in W^P$. As defined by Belkale-Kumar [BK, $\S$6] in the finite dimensional case (and extended here in Section 7 for any symmetrizable Kac-Moody case), there is a new deformed product $\odot_0$ in
$H^*(X_P, \bz)$, which is commutative and associative. Now, we are ready to state our conjecture (see Conjecture \ref{conj1}). 
\begin{conjecture} 
Let $\fg$ be any indecomposable symmetrizable Kac-Moody  Lie algebra
and let  $(\lambda_1, \dots, \lambda_s, \mu)\in P_+^{s+1}$. Assume further that none of $\lambda_j$  is $W$-invariant and $\mu-\sum_{j=1}^s \lambda_j\in Q$, where $Q$ is the root lattice of $G$. 
Then, the following are equivalent:

(a)  $(\lambda_1, \dots, \lambda_s, \mu)\in \Gamma_s$.

(b) For every standard maximal parabolic subgroup $P$ in $G^{\min}$ and every choice of
$s+1$-tuples  $(w_1, \dots, w_s, v)\in (W^P)^{s+1}$ such that $\epsilon_P^v$ occurs with coefficient $1$ in 
the deformed product
$$\epsilon_P^{w_1}\odot_0\, \cdots \,\odot_0 \epsilon_P^{w_s}
\in \bigl(H^*(X_P,\Bbb{Z}), \odot_0\bigr),$$
  the following inequality  holds:
    \[
\bigl(\sum_{j=1}^s \lambda_j(w_jx_{P})\bigr)-\mu(vx_P)\geq 0, \tag{$I^P_{(w_1,\dots, w_s,v)}$}
\]
where $\alpha_{i_P}$ is the (unique) simple root not in the Levi of $P$
and $x_P:=x_{i_P}$.
\end{conjecture}

This conjecture is motivated from its validity in the finite case due to Belkale-Kumar [BK, Theorem 22].  (For a survey of these results in the finite case, see [K$_5$].)  So far, the only evidence of its validity in the infinite dimensional case is shown for $s=2$ and $\fg$ of types $A_1^{(1)}$ and  $A_2^{(2)}$ (cf. Theorems \ref{thm7.5} and \ref{thm8.6}). In these cases, we explicitly determine $\Gamma_2$ and thereby show the validity of the conjecture. 

A positive integer  $d_o$ is called a {\em saturation factor} for $\fg$ if  for any $\Lambda$, $\Lambda'$, $\Lambda'' \in P_+$
such that  $\Lambda-\Lambda'-\Lambda''\in Q$ and 
 $L(N\Lambda)$ is a submodule of $L(N\Lambda')\otimes L(N\Lambda'')$,  for some $N\in \bz_{>0}$, then 
 $L(d_o\Lambda)$ is a submodule of $L(d_o\Lambda')\otimes L(d_o\Lambda'')$.

We prove the following result on saturation factors (cf. Corollaries \ref{cor6.4} and \ref{cor8.7}).
\begin{thm} For $A_1^{(1)}$, any integer $d_o>1$ is a saturation factor.  For $A_2^{(2)}$, $4$ is a saturation factor.
\end{thm}

The proof in these affine rank-2 cases makes  use of  basic representation theory of the Virasoro algebra (in particular, Lemma \ref{virasoro}).  
Let $\delta$ be the smallest positive imaginary root of $\fg$. To determine the saturated tensor semigroup, we show that it is enough to know the components of $L(\lambda_1)\otimes L(\lambda_2)$ which are $\delta$-maximal, i.e., the components $L(\mu)\subset L(\lambda_1)\otimes L(\lambda_2)$ such that $L(\mu+n\delta) \nsubseteq L(\lambda_1) \otimes L(\lambda_2)$ for any $n>0$. Let $m^\mu_{\lambda_1,\lambda_2}$ be  the multiplicity of $L(\mu)$ in $L(\lambda_1) \otimes L(\lambda_2)$.  If $L(\mu)$ is a $\delta$-maximal component of $L(\lambda_1)\otimes L(\lambda_2)$, then  $\sum_{n\in\bz_{\leq 0}} L(\mu+n\delta)^{\oplus m^{\mu+n\delta}_{\lambda_1,\lambda_2}}$ is a unitarizable coset module for the Virasoro algebra arising from the Sugawara construction for the diagonal embedding $\fg \hookrightarrow \fg \oplus\fg$. Proposition \ref{maximum} for  $A_1^{(1)}$ (and the analogous Proposition \ref{maxS} for $A_2^{(2)}$) determining the maximal $\delta$-components plays a crucial role in the proofs.

\vskip2ex
\noindent
{\bf Acknowledgements.} We thank Evgeny Feigin and Victor Kac for some helpful correspondences. Both the authors were partially supported by the NSF grant number DMS-1201310.
\section{Notation}\label{sec2}

We take the base field to be the field of complex numbers $\bc$. By a variety, we
 mean an algebraic variety over $\bc$, which is reduced but not necessarily irreducible.

Let $G$ be any symmetrizable Kac-Moody group over $\bc$ completed
along
 the negative roots (as opposed to completed along the positive roots as in [K$_3$,
 Chapter 6]) and $G^{\min}\subset G$ be the `minimal' Kac-Moody group  as in
 [K$_3$, \S7.4].  Let $B$ be the standard (positive) Borel
subgroup, $B^{-}$ the standard negative Borel subgroup, $H=B\cap B^{-}$ the
standard maximal torus and $W$ the Weyl group  (cf.
[K$_3$, Chapter 6]).  Let $U$ (resp. $U^-$) be the unipotent radical $[B,B]$ (resp. $[B^-,B^-]$) of
$B$ (resp. $B^-$).  Let
  \[
\bar{X} = G/B
  \]
be the `thick' flag variety which contains the standard KM-flag
variety
  \[   X = G^{\min}/B.   \]
If $G$ is not of finite type, $\bar{X}$ is an infinite
dimensional  non quasi-compact  scheme (cf. [Ka, \S4]) and $X$ is an
ind-projective variety (cf. [K$_3$, \S7.1]). The group $G^{\min}$  acts on $\bar{X}$ and $X$. 

More generally, for any 
standard parabolic subgroup $P\supset B$, define the partial flag variety
 \[   X_P = G^{\min}/P,  \]
and 
\[\bar{X}_P=G/P.\]

Recall that if $W_P$ is the Weyl group of $P$ (which is, by definition, the  Weyl
Group $W_L$ of its Levi subgroup $L$), then in each coset of $W/W_P$ we have a unique member $w$ of minimal length.
 Let $W^P$ be the set of the minimal length representatives
in the cosets of $W/W_P$.

For any $w\in W^P$, define the Schubert cell:
\[
C_w^P:=  BwP/P \subset G/P
  \]
endowed with the reduced subscheme structure.
Then, it is a locally closed subvariety of the ind-variety $G/P$ isomorphic with the affine
space $\Bbb A^{\ell(w)}, \ell(w)$ being the length of $w$ (cf. [K$_3$, $\S$7.1]). Its closure is denoted by $X^P_w$, 
which is an irreducible (projective) subvariety
of $G/P$ of dimension $\ell(w)$. We denote the point $wP\in C_w^P$ by $\dot{w}$.
We abbreviate $C^B_w, X_w^B$ by $C_w, X_w$ respectively.

Similarly, define the opposite Schubert cell
$$
C^{w}_P:={B^{-}wP/P}\subset \bar{X}_P,
$$
and the opposite Schubert variety
$$
X^{w}_P:=\overline{C^w}\subset \bar{X}_P,
$$
both endowed with the reduced subscheme structures.
Then, $X^{w}_P$ is a finite codimensional irreducible subscheme
of $\bar{X}_P$ (cf. [K$_3$, Section 7.1] and [Ka, \S4]). As above, we abbreviate $C_B^w, X^w_B$ by $C^w, X^w$ respectively.

For any integral
weight $\lambda$ (i.e., any character $e^{\lambda}$ of $H$), we have
a $G^{\min}$-equivariant line bundle $\mathcal{L}_B(\lambda)$ on $X$
associated to the character $e^{-\lambda}$ of $H$. Similarly, we have 
a $G$-equivariant line bundle $\mathcal{L}_{B^-}(\lambda)$ on $X^-:=G/B^-$
associated to the character $e^{\lambda}$ of $H$. 

 By the Bruhat decomposition
$$X_P=\sqcup_{w\in W^P}\,C_w^P,$$
the singular homology $H_*(X_P, \bz)$ of $X_P$ with integral coefficients
has a basis $\{\mu(X_w^P)\}_{w\in W^P}$, where $\mu(X_w^P)\in H_{2\ell(w)}(X_P, \bz)$ denotes the 
fundamental class of $X_w^P$. Let $\{\epsilon^w_P\}_{w\in W^P}$ be the dual 
basis of the singular cohomology $H^*(X_P, \bz)$ under the standard pairing of cohomology with homology, i.e., 
$$\epsilon^u_P(\mu(X_v^P))=\delta_{u,v},\,\,\,\text{for any} \,\,u,v\in W^P.$$
Thus, $\epsilon^w_P\in H^{2\ell(w)}(X_P, \bz)$. If $P=B$, we abbreviate 
$\epsilon^u_P$ by $\epsilon^u$. 

Let $\Delta=\{\alpha_1,\ldots,\alpha_{r}\}\subset \mathfrak{h}^{*}$ be the
set of simple roots,
$\{\alpha_1^{\vee},\ldots,\alpha^{\vee}_{r}\}\subset \mathfrak{h}$
the set of simple coroots and $\{s_1,\ldots, s_{r}\}\subset W$ the
corresponding simple reflections, where $\mathfrak{h}:=\Lie H$. Let
$\rho\in X(H)$ be any weight satisfying
$$
\rho(\alpha^{\vee}_{i})=1,\quad\text{for all}\quad 1\leq i\leq r,
$$
where $X(H)$ is the character group of $H$ (identified as a subgroup of $\fh^*$ 
via the derivative). 
When $G$ is a finite dimensional semisimple group, $\rho$ is unique,
but for a general Kac-Moody group $G$, it may not be unique.

Choose elements $x_i\in \fh$ such that 
\beqn \label{eq1} \alpha_j(x_i)=\delta_{i,j}, \,\,\,\text{for any}\,\, 1\leq i,j\leq r.
\eeqn
Observe that $x_i$ may not be unique.

Define the set of {\it dominant integral weights} 
$$P_+:=\{\lambda\in X(H): \lambda (\alpha_i^\vee)\in \bz_+ \,\forall \,
1\leq i\leq r\},$$
 and the set of {\it dominant integral regular weights} 
$$P_{++}:=\{\lambda\in X(H): \lambda (\alpha_i^\vee)\in \bz_{\geq 1} \,\forall \,
1\leq i\leq r\},$$
where  $\bz_+$ is the set of non-negative integers. The integrable highest 
weight (irreducible) modules of $G^{\min}$ are parameterized by $P_+$. For $\lambda
\in P_+$, let $L(\lambda)$ be the corresponding integrable highest weight (irreducible) $G$-module 
with highest weight $\lambda$. 

\section{Necessary Inequalities for the Saturated Tensor Semigroup}\label{sec3}
Fix a positive integer $s$ and define the {\em saturated tensor semigroup} $\Gamma_s=\Gamma_s(G)$:
\beqn
\Gamma_s:= \{(\lambda_1, \dots, \lambda_s,\mu)\in P_+^{s+1}: \exists\,
N>1 \,\,\text{with}\,\, L(N\mu)\subset L(N\lambda_1)\otimes \dots \otimes L(N\lambda_s)\}.
\eeqn
It is indeed a semigroup by the anlogue of the Borel-Weil theorem for the Kac-Moody case (see the identity \eqref{ne3.3}
in the proof of Theorem \ref{thm1}).
We give a certain set of inequalities satisfied by $\Gamma_s$. But, we first recall some basic results about the Hilbert-Mumford index.

\begin{definition}\label{git} Let $S$ be any (not necessarily reductive) algebraic group
acting on a  (not necessarily projective) variety  $\exx$ and let  $\elal$ be
an $S$-equivariant line bundle on $\exx$. Let $O(S)$ be the set of all one parameter
subgroups (for short OPS) in $S$.
 Take any $x\in \exx$ and
 $\delta \in O(S)$ such that the limit
  $\lim_{t\to 0}\delta(t)x$
exists in $\exx$ (i.e., the morphism ${\delta}_x:\Bbb{G}_m\to \exx$ given by
$t\mapsto \delta(t)x$ extends to a morphism $\tilde{\delta}_x : \Bbb{A}^1\to \exx$).
Then, following Mumford, define a number $\mu^{\elal}(x,\delta)$ as follows:
Let $x_o\in \exx$ be the point  $\tilde{\delta}_x(0)$. Since $x_o$ is $\Bbb{G}_m$-invariant
via $\delta$, the fiber of  $\elal$ over $x_o$ is a
$\Bbb{G}_m$-module; in particular, it is given by a character of $\Bbb{G}_m$. This  integer is defined as  $\mu^{\elal}(x,\delta)$.
\end{definition}

We record the following standard properties of $\mu^{\elal}(x,\delta)$ (cf.
 [MFK, Chap. 2, $\S$1]):
\begin{proposition}\label{propn14} For any $x\in \exx$ and $\delta \in O(S)$ such that $\lim_{t\to 0}\delta(t)x$
exists in $\exx$, we have the following (for any $S$-equivariant line bundles
$\elal, \elal_1, \elal_2$):
\begin{enumerate}
\item[(a)]
$\mu^{\elal_1\otimes\elal_2}(x,\delta)=\mu^{\elal_1}(x,\delta)+\mu^{\elal_2}(x,\delta).$
\item[(b)] If there exists $\sigma\in H^0(\exx,\elal)^S$ such that $\sigma(x) \neq 0$, then  $\mu^{\elal}(x,\delta)\geq 0.$
\item[(c)] If $\mu^{\elal}(x,\delta)=0$, then any element of $H^0(\exx,\elal)^S$
which does not vanish at $x$ does not vanish at $\lim_{t\to 0}\delta(t)x$ as well.
\item[(d)] For any $S$-variety $\exx'$ together with an $S$-equivariant morphism $f:\exx'\to \exx$ and any $x'\in \exx'$ such that  $\lim_{t\to 0}\delta(t)x'$
exists in $\exx'$, we have
$\mu^{f^*\elal}(x',\delta)=\mu^{\elal}(f(x'),\delta).$
\item[(e)] (Hilbert-Mumford criterion) Assume that $\exx$ is projective, $S$ is
 connected and reductive
and $\elal$ is ample. Then, $x\in\exx$ is semistable (with respect to $\elal$) if
and only if $\mu^{\elal}(x,\delta)\geq 0$, for all $\delta\in O(S)$.

In particular, if $x\in \exx$ is semistable and $\delta$-fixed, then
$\mu^{\elal}(x,\delta)= 0$.
\end{enumerate}
\end{proposition}

The following theorem is one of our main results giving a collection of necessary inequalities defining the semigroup
$\Gamma_s$.
  \begin{theorem}  \label{thm1} Let $G$ be any symmetrizable Kac-Moody group and let  $(\lam_1, \cdots ,\lam_s, \mu )\in \Gamma_s$.
  Then, for any $u_1, \dots ,u_s, v\in W$ such that $n^v_{u_1, \dots, u_s}\neq 0$, where
$$\eps^{u_1} \cdots \eps^{u_s}
= \sum_w n^w_{u_1, \dots, u_s}\,\eps^w\in H^*(X, \bz),$$ we have
  \[
\bigl(\sum^s_{j=1} \lam_j (u_jx_i)\bigr) - \mu (vx_i) \geq 0, \quad\text{ for any }x_i,
  \]
where $x_i$ is defined by the equation \eqref{eq1}.
  \end{theorem}

   \begin{proof}
Let
  \[
Z := \bigl\{ (\bar{g}_1, \dots ,\bar{g}_s)\in {(X^-)}^{s}: g_1X^{u_1}
\cap \cdots \cap g_sX^{u_s}\cap X_v \neq \emptyset\bigr\} ,
  \]
where $X^-:=G/B^-$ and $\bar{g}_j= g_jB^-$. 
 Then, $Z$ contains a nonempty open set by Proposition \ref{prop5}.
(In fact, by Proposition \ref{prop5}, $Z = (X^-)^{s}$, but we do not need this stronger result.)

Take a nonzero $\sig\in H^0 \bigl( (X^-)^{s}\times X, \cl^N
\bigr)^{G^{\min}}$, where
$$\cl := \cl_{B^-}(\lam_1)\boxtimes \cdots\boxtimes \cl_{B^-}(\lam_s)\boxtimes
\cl_B (\mu ).$$
Such a nonzero $\sigma$ exists, for some $N>0$, since by [K$_3$, Corollary 8.3.12(a) and Lemma 8.3.9], 
\begin{align} \label{ne3.3}
H^0 \bigl( (X^-)^{s}\times X, \cl^N
\bigr)^{G^{\min}}&\simeq \Hom_{G^{\min}}\bigl( L(N\lambda_1)^\vee\otimes \dots \otimes L(N\lambda_s)^\vee\otimes L(N\mu), \bc \bigr)\notag\\
&\simeq \Hom_{G^{\min}}\bigl( L(N\mu), [L(N\lambda_1)^\vee\otimes \dots \otimes L(N\lambda_s)^\vee]^*\bigr)\notag\\
&\simeq \Hom_{G^{\min}}\bigl( L(N\mu), [L(N\lambda_1)^\vee\otimes \dots \otimes L(N\lambda_s)^\vee]^\vee \bigr)\notag\\
&\simeq \Hom_{G^{\min}}\bigl( L(N\mu), L(N\lambda_1)\otimes \dots \otimes L(N\lambda_s) \bigr)\notag\\
&\neq 0,
\end{align}
since $(\lambda_1, \dots, \lambda_s,\mu)\in \Gamma_s$, where, for a $G^{\min}$-module $M$, $M^\vee$ denotes the direct sum of the $H$-weight spaces of the full dual module $M^*$. 

Pick $(\bar{g}_1, \dots ,\bar{g}_s)\in Z$ such that $\sig (\bar{g}_1,
\dots ,\bar{g}_s, \bar{1})\neq 0$, where  
$\bar{1} =1\cdot B$.  Since $(\bar{g}_1, \dots ,\bar{g}_s)\in Z$, there exists $u'_1 \geq u_1, \cdots , u'_s \geq u_s$ and $v' \leq v$
such that $g_1C^{u'_1} \cap \cdots \cap g_s C^{u'_s}
 \cap C_{v'}$ is nonempty. Now, pick $g\in G^{\min}$ such that
    \beqn \label{e101}
gB \in g_1C^{u'_1} \cap \cdots \cap g_s C^{u'_s}
 \cap C_{v'}.
  \eeqn
By Proposition \ref{propn14}, for any $\delta \in O(G^{\min})$, $\mu^{\cl}(\bar{x}, \del
(t))\geq 0$, where $\bar{x} = (\bar{g}_1, \dots ,\bar{g}_s, \bar{1})$
(since $\sig (\bar{x})\neq 0$).  By the following Lemma \ref{lem2},
applied to the OPS $\delta (t)=gt^{x_i}g^{-1}$, we get
  \beqn \label{eqn02}
\bigl(\sum^s_{j=1} \lam_j (u'_jx_i)\bigr) - \mu (v' x_i) \geq 0.  
  \eeqn
But, by [K$_3$, Lemma 8.3.3],
  \[  (u'_j)^{-1}\lam_j \leq u_j^{-1} (\lam_j).   \]
Thus,
  \[   \lam_j(u'_j x_i) \leq \lam_j(u_jx_i).   \]
Similarly,
  \[   \mu (v'x_i) \geq \mu (v x_i).   \]
Thus, from \eqref{eqn02}, we get
  \[   \bigl(\sum_{j=1}^s \lam_j(u_jx_i) \bigr)- \mu (vx_i) \geq 0.   \]

  This proves the theorem. 
  \end{proof}

  \begin{lemma} \label{lem2} Let $g\in G^{\min}$ be as in the equation \eqref{e101}.
Consider the one parameter subgroup
 $
\del (t) = gt^{x_i} g^{-1}\in O(G^{\min}).
  $ Then,

(a) $\mu^{\cl_{B^-}(\lam_j)} (g_jB^-, \del (t)) = \lam_j(u_j' x_i)$.

(b) $\mu^{\cl_B(\mu )}(1\cdot B, \del (t)) = -\mu (v'x_i)$.
    \end{lemma}

  \begin{proof}  (a) 
 $\mu^{\cl_{B^-}(\lam_j)} (g_jB^-, \del (t))= \mu^{\cl_{B^-}
  (\lam_j)}(g^{-1}g_jB^-, t^{x_i})$.\\
By assumption, $g^{-1}_jg\in U^- u'_jB$.  Write
  \[
g_j^{-1}g = b_j^-u'_jp_j, \quad\text{ for some } \,b_j^-\in U^-,\, p_j \in B.
  \]
Thus,
  \[    1 = g^{-1}g_jb_j^- u_j' p_j.     \]
Let
  \[    b_j(t) = b_j^- u'_j t^{-x_i} (u'_j)^{-1} (b_j^-)^{-1} \in B^- .    \]
Then,
   \beqn\label{eqn01}
t^{x_i} g^{-1}g_jb_j(t) = t^{x_i}p_j^{-1} t^{-x_i} (u'_j)^{-1} (b_j^-)^{-1}.
  \eeqn
Consider the $G_m$-invariant section (via $t^{x_i}$) of $\cl_{B^-}(\lam_j):$
  \begin{align*}
\hat{\sig}(t) &= \bigl( t^{x_i}\, g^{-1}g_j, 1\bigr) \mod B^-\\
&= \bigl( t^{x_i}\, g^{-1}g_jb_j(t), \lam_j (b_j(t)^{-1})\bigr) \mod B^- .
  \end{align*}
Clearly,  $\lt_{t\to 0} \,t^{x_i}\, g^{-1}g_jb_j(t)$ exists in $G$ by \eqref{eqn01}.

Now,
  \begin{align*}
\lam_j \bigl( b_j(t)^{-1}\bigr) &= \lam_j\bigl( b_j^- u'_j t^{x_i} (u'_j)^{-1}
(b_j^-)^{-1}\bigr)\\
  &= \lam_j \bigl( t^{u_j' x_i}\bigr) .
    \end{align*}
This gives
  \[
\mu^{\cl_{B^-}(\lam_j)} (g_jB^-, \del (t)) = \lam_j (u'_j(x_i)) .
  \]

This proves the (a) part of the lemma.
\vskip1ex

  (b)  $\mu^{\cl_B(\mu )}(1\cdot B, \del (t)) = \mu^{\cl_B(\mu )}
  (g^{-1}B, t^{x_i})$.  By assumption,
  \[
 g\in B v'\cdot B.
   \]
Write
  \[
g = bv'p, \quad\text{for }b\in U, p\in B.
  \]
Thus,
  \[    1 = g^{-1} b v'p.   \]
Let
  \[    b(t) = bv't^{-x_i}(v')^{-1}b^{-1} \in B.   \]
Now,
  \[
t^{x_i}g^{-1} b(t) = t^{x_i}p^{-1} t^{-x_i}(v' )^{-1}b^{-1}.
  \]
Thus,
  \[
\lt_{t\to 0} \,t^{x_i}g^{-1}b(t) \text{ exists in } G^{\min}.
  \]

Consider the $G_m$-invariant section (via $t^{x_i}$)
  \begin{align*}
\hat{\sig}(t) &= (t^{x_i}g^{-1}, 1) \mod B\\
&= \bigl( t^{x_i}g^{-1}b(t), \mu (b(t))\bigr) \mod B.
  \end{align*}
Now,
  \begin{align*}
\mu (b(t)) &= \mu (bv' t^{-x_i} (v')^{-1}b^{-1})\\
&= \mu (t^{-v' x_i}).
  \end{align*}
This gives
  \[
\mu^{\cl_B(\mu )}(1\cdot B, \del (t)) = -\mu (v' (x_i)).
  \]

  This proves the (b)-part and hence the lemma is proved.
    \end{proof}

  \begin{definition} \label{n2.1}
  For a quasi-compact scheme $Y$, an $\co_{Y}$-module $\cs$ is called {\it coherent}
  if it is finitely presented as an $\co_{Y}$-module and any $\co_{Y}$-submodule of finite
  type admits a
finite presentation.

  An $\co_{\bar{X}}$-module $\cs$ is called {\it coherent} if
   $\cs_{|V^S}$ is a
coherent $\co_{V^S}$-module for any finite ideal $S\subset W$ (where a subset $S\subset W$
is called an {\it ideal} if
 for $x\in S$ and $y\leq x\Rightarrow y\in S$), where $V^S$ is the quasi-compact open subset
 of $\bar{X}$ defined by
 $$V^S = \bigcup_{w\in S} wU^- B/B.$$
 Let $K^0(\bar{X})$ denote the Grothendieck group of
 coherent $\co_{\bar{X}}$-modules $\cs$.

 Similarly,
define $K_0(X) := \lim_{n\to\infty} K_0(X_n)$, where $\{
X_n\}_{n\geq 1}$ is the filtration of $X$ giving the ind-projective
variety structure (i.e., $X_n = \bigcup_{\ell (w)\leq n} C_w$) and
$K_0(X_n)$ is the Grothendieck group of  coherent
sheaves on the projective variety $X_n$.

We also define
  \[
K^{\optop}(X) := \Invlt_{n\to\infty} K^{\optop}(X_n),
  \]
where $K^{\optop}(X_n)$ is the topological $K$-group of the
 projective variety $X_n$.

Let $*:K^{\optop}(X_n)\to K^{\optop}(X_n)$ be the involution induced from
the operation which takes a vector bundle to its dual. This,
 of course, induces the involution $*$ on $K^{\optop}(X)$.

For any $w\in W$,
  \[  [\co_{X_w}] \in K_0(X).  \]
  \end{definition}

  \begin{lemma}  $\bigl\{ [\co_{X_w}]\bigr\}_{w\in W}$ forms a basis of $K_0(X)$ as a $\bz$-module.
  \end{lemma}

  \begin{proof} By [CG, \S 5.2.14 and Theorem 5.4.17], the result follows.
  \end{proof}

  For $u\in W$, by [KS, \S 2], $\co_{X^u}$ is a coherent $\co_{\bar{X}}$-module.
  In particular, $\co_{\bar{X}}$ is a coherent $\co_{\bar{X}}$-module.

Define a pairing
$$
\langle \, ,\, \rangle : K^0(\bar{X}) \otimes K_0(X) \to \bz,\,\,
\langle[\cs], [\cf]\rangle  = \sum_i (-1)^i \chi \bigl (X_n, \tor_
i^{\co_{\bar{X}}}
 (\cs,\cf ) \bigr),$$
if $\cs$ is a coherent sheaf on $\bar{X}$ and $\cf$
is a coherent sheaf on ${X}$ supported in $X_n$ (for
some $n$), where $\chi$ denotes the 
Euler-Poincar\'{e} characteristic.
Then, as in [K$_4$, Lemma 3.4], 
  the above pairing is well defined.
 
By [KS, Proof of Proposition 3.4], for any $u\in W$,
  \beqn\label{eq1.0}
\ext^k_{\co_{\bar{X}}} (\co_{X^u}, \co_{\bar{X}}) =0 \quad\forall k\neq \ell (u).
  \eeqn
Define the sheaf
  \[
\om_{X^u} := \ext^{\ell (u)}_{\co_{\bar{X}}}
\bigl(\co_{X^u}, \co_{\bar{X}} \bigr)\otimes\cl (-2\rho ),
  \]
  which, by the analogy with the Cohen-Macaulay (for short CM) schemes of finite type, will be called
  the {\it dualizing sheaf} of $X^u$.

Now, set the  sheaf on $\bar{X}$
  \begin{align*}
\xi^u &:= \cl (\rho )\om_{X^u} \\
&= \cl (-\rho ) \ext^{\ell (u)}_{\co_{\bar{X}}}
(\co_{X^u}, \co_{\bar{X}} ).
  \end{align*}
Then,  as proved in  [K$_4$, Proposition 3.5],  for any $u,w\in W$,
\beqn\label{e106}
\langle[\xi^u], [\co_{X_w}]\rangle = \delta_{u,w}.
\eeqn
With these preliminaries, we are ready to prove the following result.
  \begin{proposition}  \label{prop5} With the notation as in the proof of Theorem \ref{thm1},
$Z = (X^-)^{s}$, if  $\eps^v $ occurs in $\eps^{u_1}
\cdots\eps^{u_s}$ with
 nonzero coefficient.
  \end{proposition}

  \begin{proof}  We give the proof in the case $s=2$.  The proof for general
 $s$ is similar.

For $u,v\in W$, express
  \[
\eps^u\eps^v = \sum_{ \substack{w\\ \ell (w)=\ell (u)+\ell (v)} } n^w_{u,v} \eps^w.
  \]
Express the product in topological $K$-theory $K^{\optop}(X)$ of $X=G^{\min}/B$:
  \[
\psi^u_o\psi^v_o = \sum_{\ell (w)\geq\ell (u)+\ell (v)} m^w_{u,v} \psi^w_o,
  \]
where $\psi^w := *\tau^{w^{-1}}$ ($\tau^w$ being the Kostant-Kumar `basis'
of $K^{\optop}_H(X)$ as in [KK, Remark 3.14]) and $\{\psi^{w}_o\}_{w \in W}$ is the corresponding `basis' of 
$K^{\optop}(X)\simeq \bz\otimes_{R(H)}\,K^{\optop}_H(X),$  cf. [KK, Proposition 3.25]). 

Then, by [KK, Proposition 2.30],
  \beqn\label{e102}
n_{u,v}^w = m_{u,v}^w,  \quad\text{if }\ell (w) = \ell (u)+\ell (v).
  \eeqn
Let $\Delta: X \to X\times X$ be the diagonal map. Then, by [K$_4$, Proposition 4.1] and the identity \eqref{e106}, for any $u,v,w \in W$, $g_1,g_2\in G^{\min}$,
  \begin{align*}
m^w_{u,v} &= \langle [\xi^u\boxtimes\xi^v], [\Del_*\co_{X_w}]\rangle \\
&= \langle [\xi^u\boxtimes\xi^v], [(g_1^{-1}, g_2^{-1}) \cdot (\Del_* \co_{X_w})]\rangle,
  \end{align*}
since $[(g_1^{-1}, g_2^{-1})\cdot\Del_*\co_{X_w}]= [\Del_*\co_{X_w}]$ as elements
of $K_0(X\times X)$.  Thus,
 \begin{align}\label{e103}
 m^w_{u,v} &= \langle [\xi^u\boxtimes\xi^v], [(g_1^{-1}, g_2^{-1}) \cdot (\Del_*\co_{X_w})]\rangle
 \\ &:= \sum_i (-1)^i \chi (\bar{X}\times \bar{X}, \tor_i^{\co_{\bar{X}\times \bar{X}}} \Bigl(\xi^u\boxtimes\xi^v,(g_1^{-1}, g_2^{-1}) \cdot (\Del_* \co_{X_w})\Bigr) .\notag
  \end{align}

Now, by definition, the support of $\xi^u$ is contained in $X^u$ and hence the
support of the sheaf
  \[
\cs_i := \tor_i^{\co_{\bar{X}\times \bar{X}}} \bigl( \xi^u\boxtimes\xi^v, (g_1^{-1}, g_2^{-1})\cdot \Del_* \co_{X_w}\bigr)
  \]
is contained in
  \beqn\label{e104}
X^u\times X^v \cap \bigl((g_1^{-1}, g_2^{-1})\cdot \Del (X_w)\bigr),
  \eeqn
which is empty if
  \beqn\label{e105}
(g_1X^u) \cap (g_2X^v) \cap X_w = \emptyset .
  \eeqn
Thus, if the equation \eqref{e105} is true, then the Tor sheaf $\cs_i =0$ $\forall i\geq 0$.  Thus, if 
the equation \eqref{e105} is satisfied, 
  \[   m_{u,v}^w =0.   \]
Now, assume that $\ell (w) = \ell (u)+\ell (v)$.  Then, by the equation \eqref{e102},
  \[  n^w_{u,v} =0, \quad\text{ if the equation \eqref{e105} is satisfied}.   \]
But, since by assumption, $n^w_{u,v} \neq 0$, we see that
  \[
(g_1X^u) \cap (g_2X^v )\cap X_w \neq \emptyset , \;\text{ for any } g_1,g_2\in G^{\min}.
  \]
But since $G^{\min}/(G^{\min} \cap B^-) \simto X^-$, we get the proposition.
  \end{proof}

\section{Tensor Product Decomposition for Affine Kac-Moody Lie Algebras}

\subsection{The Virasoro Algebra}

We recall the definition of the Virasoro algebra and its basic representation theory, which we need. 
The {\it Virasoro algebra} $\mathrm{Vir}$ has a basis $\{C,\, L_{n}\;:\; n\in\mathbb{Z}\}$
over $\bc$ 
and the Lie bracket is given by 
\[[L_{m},L_{n}]=(m-n)L_{m+n}+\frac{1}{12}(m^{3}-m)\delta_{m,-n}C\,\,
\text{and}\, [\mathrm{Vir,C]=0}.\] 

Let $\Vir_0:= \mathbb{C}L_{0}\oplus\mathbb{C}C$. Then,  a Vir module
$V$ is said to be a {\it highest weight representation} if there exists a $\Vir_0$-eigenvector $v_o\in V$
such that $L_{n}v_o=0$ for $n\in\mathbb{Z}_{>0}$ and $U(\bigoplus_{n<0}\mathbb{C}L_{n})v_o=V$.
Such a $V$ is said to have {\it highest weight}   $\lambda\in \Vir_0^{*}$
if $Xv_o=\lambda(X)v_o$, for all $X\in \Vir_0$. (It is easy to see that such a $v_o$ is unique up to a scalar multiple and hence
$\lambda$ is unique.)
The irreducible
highest weight representations of Vir are in 1-1 correspondence with elements
of $\Vir_0^{*}$ given by the highest weight. 
Denote the basis of $\Vir_0^*$ dual to the basis $\{L_{0},C\}$ of $\Vir_0$ as $\{h,z\}$. 
For any $\mu\in \Vir_0^*$, denote the $\mu$-th weight space 
of $V$ by $V_\mu$, i.e.,
\[V_\mu:=\{v\in V: X\cdot v=\mu(X)v\,\, \forall X\in \Vir_0\}.\]

Define a Vir module $V$ to be {\it unitarizable} if there exists a positive
definite Hermitian form $(\cdot\,,\,\cdot)$ on $V$ so that $(L_{n}v\,,\, w)=(v\,,\, L_{-n}w)$
for all $n\in\mathbb{Z}$ and $(Cv\,,\, w)=(v\,,\, Cw)$. It is easy
to see that if $M$ is a $\Vir$-submodule of $V$, then $M^{\perp}$
is also a submodule. Hence, any unitarizable representation of Vir
is completely reducible. Note that for a unitarizable highest weight Vir-representation
$V$ with highest weight $\lambda$, if $v_o$ is a highest weight vector,
then 
\beqn \label{e4.2}
0\leq(L_{-n}v_o\,,\, L_{-n}v_o)=(L_{n}L_{-n}v_o\,,\, v_o)=(2n\lambda(L_{0})+\frac{1}{12}(n^{3}-n)\lambda(C))(v_o\,,\, v_o)
\eeqn
 for all $n>0$. Therefore, both $\lambda(L_{0})$ and $\lambda(C)$
must be nonnegative real numbers.

\begin{lem}\label{virasoro}
Let $V$ be a unitarizable,  highest weight (irreducible) representation of $Vir$
with highest weight $\lambda$. 

(a) If $\lambda(L_0)\neq 0$, then $V_{\lambda+nh}\neq 0$, for any $n\in \bz_+$. 

(b) If 
$\lambda(L_0)= 0$ and $\lambda (C)\neq 0$, then $V_{\lambda+nh}\neq 0$, for any $n\in \bz_{>1}$ and 
$V_{\lambda+h} =0$.

(c)  If 
$\lambda(L_0)= \lambda (C) = 0$, then $V$ is one dimensional. 
\end{lem}
\begin{proof}
If  $\lambda(L_0)\neq 0$, then by the equation \eqref{e4.2} (since both of $\lambda (L_0)$ and $\lambda(C)\in \br_+$), 
$L_{-n}v_o\neq 0$, for any $n\in \bz_+$. 

If  $\lambda(L_0)=0$ and $\lambda(C)\neq 0$,  then again by the equation \eqref{e4.2}, 
$L_{-n}v_o\neq 0$, for any $n\in \bz_{>1}$. Also, $L_{-1}v_o=0$. 

 If $\lambda(L_0)=\lambda(C)= 0$, then (by the equation \eqref{e4.2} again), 
$L_{-n}v_o= 0$, for any $n\in \bz_{\geq 1}$. This shows that $V$ is one dimensional.
\end{proof}
\subsection{Tensor product decomposition: A general method}

Let $\mathfrak{g}$ be the untwisted affine Kac-Moody Lie algebra associated to a  finite dimensional simple 
Lie algebra $\overset{\circ}{\mathfrak{g}}$, i.e.,
\[\fg=\bigl(\overset{\circ}{\mathfrak{g}}\otimes \bc[t,t^{-1}]\bigr) \oplus \bc c\oplus \bc d.\]
Let $\overset{\circ}{\mathfrak{h}}$ be a Cartan subalgebra of $\overset{\circ}{\mathfrak{g}}$. Then, 
\[\fh:=\overset{\circ}{\mathfrak{h}}\otimes 1\oplus\bc c\oplus \bc d\]
is the standard Cartan subalgebra of $\fg$.
 Let $\delta\in \fh^*$ be the smallest positive imaginary root of 
$\fg$ (so that the positive imaginary roots of $\fg$ are precisely $\{n\delta, n\in \bz_{\geq 1}\}$).  Then,
$\delta$ is given by $\delta_{|\overset{\circ}{\mathfrak{h}}\oplus \bc c}\equiv 0$ and $\delta(d)=1$. 
For any $\Lambda \in P_+$, let $P(\Lambda)$ 
be the set of weights of $L(\Lambda)$ and let $P^o(\Lambda)$ be the set of $\delta$-maximal weights of $L(\Lambda)$, i.e.,
\[
P^o(\Lambda)=\left\{ \lambda\in\mathfrak{h}^{*}:\lambda \in P(\Lambda) \,\,\text{but}\,\, \lambda+n\delta \notin P(\Lambda)\,\,
\text{for any}\,\,n>0\right\}.
\]
For any $\lambda \in X(H)$, define the $\delta$-{\it character of $L(\Lambda)$ through} $\lambda$ by
\[c_{\Lambda,\lambda}=\sum_{n\in\mathbb{Z}}\dim L(\Lambda)_{\lambda+n\delta}\,e^{n\delta}.\]
Since $\delta$ is $W$-invariant,
\beqn \label{e4.1}
c_{\Lambda,\lambda}=c_{\Lambda, w\lambda}, \,\,\text{for any}\, w\in W.
\eeqn
Moreover, $P^o(\Lambda)$ is $W$-stable.
 It
is obvious that 
\beqn \label{e13}
ch\, L(\Lambda)=\sum_{\lambda\in P^o(\Lambda)}c_{\Lambda,\lambda}e^{\lambda}.
\eeqn
By [K$_3$, Exercise 13.1.E.8], for any $\lambda\in P(\Lambda')$ and  $\Lambda''\in P_+$, $\Lambda''+\lambda+\rho$ belongs to the Tits cone. Hence,
there exists $v\in W$ such that $v^{-1}(\Lambda''+\lambda+\rho)\in P_+$. Moreover, if $\Lambda''+\lambda+\rho$ has nontrivial $W$-isotropy, then its isotropy group must contain a reflection (cf. [K$_3$, Proposition 1.4.2(a)]). Thus, for such a $\lambda\in P(\Lambda')$, i.e., if $\Lambda''+\lambda+\rho$ has nontrivial $W$-isotropy,
\beqn\label{e14} \sum_{w\in W}\,\varepsilon(w) e^{w(\Lambda''+\lambda+ \rho)}=0.\eeqn
Define 
$$ \bar{P}_+:=\{\Lambda \in P_+: \Lambda(d)=0\}.$$
For any $m\in \bz_+$, let 
\[P_+^{(m)}:=\{\Lambda\in P_+: \Lambda (c)=m\},\]
and let 
\[\bar{P}_+^{(m)}:=\bar{P}_+\cap  P_+^{(m)}.\]
 Then, 
${\bar{P}}_+^{(m)}$ provides a set of representatives in $P_+^{(m)}$ mod $(P_+\cap \bc\delta)$.  

For any $\Lambda, \Lambda',\Lambda''\in P_+$, define
\begin{align*} T_{\Lambda}^{\Lambda',\Lambda''}=\{&\lambda \in P^o(\Lambda'): \exists v_{\Lambda,\Lambda'', \lambda}\in W\,\,\text{and}\, 
S_{\Lambda,\Lambda'', \lambda}\in \bz \,\,\text{with}\\ 
&\, \lambda+\Lambda''+\rho=v_{\Lambda,\Lambda'', \lambda}(\Lambda+\rho)+ 
S_{\Lambda,\Lambda'', \lambda}\delta\}.
\end{align*}
Observe that since $\Lambda+\rho +n\delta \in P_{++}$ for any $n\in \bz$, such a $v_{\Lambda,\Lambda'', \lambda}$ and $S_{\Lambda,\Lambda'', \lambda}$ are unique by [K$_3$, Proposition 1.4.2 (a), (b)]
(if they exist). Also, observe that 
\beqn \label{eq1001} T_{\Lambda}^{\Lambda',\Lambda''}=\emptyset,\,\,\text{unless}\, \Lambda (c)=\Lambda'(c)+\Lambda''(c)\,\,\,\text{and}\,\, 
\Lambda'+\Lambda''-\Lambda\in Q,\eeqn
where $Q$ is the root lattice of $\fg$.
\begin{prop} \label{tensor} For any $\Lambda'$ and $\Lambda''\in P_+$, 
\[ch\,\bigl( L(\Lambda')\otimes L(\Lambda'')\bigr) = \sum_{\Lambda\in \bar{P}_{+}^{(m)}}ch\, L(\Lambda)\bigl(\sum_{\lambda\in T_{\Lambda}^{\Lambda',\Lambda''}}\varepsilon(v_{\Lambda,\Lambda'',\lambda})c_{\Lambda',\lambda}e^{S_{\Lambda,\Lambda'',\lambda}\delta}\bigr),\]
where $m:=\Lambda'(c)+\Lambda''(c)$. 

Moreover, $\sum_{\lambda\in T_{\Lambda}^{\Lambda',\Lambda''}}\varepsilon(v_{\Lambda,\Lambda'',\lambda})c_{\Lambda',\lambda}e^{S_{\Lambda,\Lambda'',\lambda}\delta}$ is the character of a unitary representation (though, in general, not irreducible) of the Virasoro algebra 
$\mathrm{Vir}$  with central charge 
$$\dim \overset{\circ}{\mathfrak{g}}\cdot\bigl(\frac{m'}{m'+g}+\frac{m''}{m''+g}-\frac{m}{m+g}\bigr),$$
where $m':=\Lambda'(c), m'':=\Lambda''(c)$ and $g$ is the dual Coxeter number of $\overset{\circ}{\mathfrak{g}}$.
\end{prop}
\begin{proof}
By the Weyl-Kac character formula (cf. [K$_3$, Theorem 2.2.1]) and the identity \eqref{e13}, for any $\Lambda', \Lambda''\in P_+$, 
\begin{align*}
\left(\sum_{w\in W}\varepsilon(w)e^{w\rho}\right)&\cdot  ch\, L(\Lambda')\cdot ch\, L(\Lambda'')\\
&=\left(\sum_{\lambda\in P^o(\Lambda')}c_{\Lambda',\lambda}e^{\lambda}\right)\cdot\left(\sum_{w\in W}\varepsilon(w)e^{w(\Lambda''+\rho)}\right)\\
&=\sum_{\lambda\in P^o(\Lambda')}c_{\Lambda',\lambda}\sum_{w\in W}\varepsilon(w)e^{w(\Lambda''+\lambda+\rho)},
\,\, \text{by}\, \eqref{e4.1}\\
&=  \sum_{\Lambda\in \bar{P}_{+}^{(m)}}\sum_{\lambda\in T_{\Lambda}^{\Lambda',\Lambda''}}c_{\Lambda',\lambda}\sum_{w\in W}\varepsilon(w)e^{w(v_{\Lambda,\Lambda'',\lambda}(\Lambda+\rho))+S_{\Lambda,\Lambda'',\lambda}\delta},\,\,\text{by}\, \eqref{e14}\\
 &= \sum_{\Lambda\in \bar{P}_{+}^{(m)}}\sum_{\lambda\in T_{\Lambda}^{\Lambda',\Lambda''}}c_{\Lambda',\lambda}\sum_{w\in W}\varepsilon(w)\varepsilon(v_{\Lambda,\Lambda'',\lambda})e^{w(\Lambda+\rho)}e^{S_{\Lambda,\Lambda'',\lambda}\delta}\\
 &=  \sum_{\Lambda\in \bar{P}_{+}^{(m)}}\sum_{w\in W}\varepsilon(w)e^{w(\Lambda+\rho)}\sum_{\lambda\in T_{\Lambda}^{\Lambda',\Lambda''}}\varepsilon(v_{\Lambda,\Lambda'',\lambda})c_{\Lambda',\lambda}e^{S_{\Lambda,\Lambda'',\lambda}\delta}.
\end{align*}
Thus,
\[ch\, \bigl(L(\Lambda')\otimes L(\Lambda'')\bigr) = \sum_{\Lambda\in \bar{P}_{+}^{(m)}}ch\, L(\Lambda)\bigl(\sum_{\lambda\in T_{\Lambda}^{\Lambda',\Lambda''}}\varepsilon(v_{\Lambda,\Lambda'',\lambda})c_{\Lambda',\lambda}e^{S_{\Lambda,\Lambda'',\lambda}\delta}\bigr).
\]
To prove the second part of the proposition, use [KR, Proposition 10.3].
This proves the proposition.
\end{proof}
\begin{remark} For an affine Kac-Moody Lie algebra $\fg$, if we consider the tensor product decomposition of $L(\Lambda')\otimes L(\Lambda'')$ with respect to the derived subalgebra $\fg'$ (i.e., without the $d$-action), then the components $L(\Lambda)$ are precisely of the form 
$\Lambda\in \Lambda'+\Lambda''+\overset{\circ}{Q}$, where $\overset{\circ}{Q}$ is the root lattice of $\overset{\circ}{\mathfrak{g}}$
(cf. [KW]). Thus, the determination of the eigen semigroup and the saturated eigen semigroup is fairly easy for $\fg'$.
\end{remark}

Let $\theta=\sum_{i=1}^\ell h_i\alpha_i$ be the highest root of $\overset{\circ}{\mathfrak{g}}$ (with respect to a choice of the positive roots), written as a linear combination of the simple roots $\{\alpha_1, \dots, \alpha_\ell\}$ of $\overset{\circ}{\mathfrak{g}}$. Let
\[ S:= \{\sum_{i=0}^\ell \, n_i\alpha_i: n_i\geq 0 \,\,\,\text{for any}\, i\,\,\,\text{and}\, 0\leq n_i< h_i \,\,\text{for some}\,\, 0\leq i\leq \ell \},\]
where $h_0:=1$.
\begin{prop} \label{prop4.1} Let $\fg$ be an untwisted affine Kac-Moody Lie algebra as above. Then, 
for any $\Lambda \in P_+$ with $\Lambda (c)>0$, 
\[P^o(\Lambda)_+= S(\Lambda)\cap P_+,\]
where $P^o(\Lambda)_+:=P^o(\Lambda)\cap P_+$ and
$S(\Lambda)=\{\Lambda-\beta: \beta \in S\}.$
\end{prop}
\begin{proof}
Take $\lambda \in S(\Lambda)$. Then, for any $n\geq 1$, 
\[
\Lambda-(\lambda+n\delta)=\bigl(\sum_{i=0}^\ell\, n_{i}\alpha_{i}\bigr)-n\delta=(n_{0}-n)\alpha_{0}+ \sum_{i=1}^\ell\,(n_{i}-nh_{i})\alpha_{i},
\]
 since $\alpha_{0}:=\delta-\theta$. Now, the coefficient of some $\alpha_{i}$
in the above sum is  negative, for any positive $n$, since $\lambda \in S(\Lambda)$. Thus, 
 $\lambda+n\delta$
could not be a weight of $L(\Lambda)$  for any positive $n$. Therefore, if $\lambda \in P(\Lambda)\cap S(\Lambda)$, then   it  is $\delta$-maximal.

By [Kac, Proposition 12.5(a)], if $\Lambda(c)\neq 0$, then $ S(\Lambda) \cap P_+\subset P(\Lambda).$ Therefore,
 $ S(\Lambda) \cap P_+\subset P^o(\Lambda)_+.$

Conversely, take $\lambda \in  P^o(\Lambda)_+$. Then,  $\lambda \in  P(\Lambda) \cap P_+$ and 
$\lambda+\delta \notin P(\Lambda)$. Express $\lambda=\Lambda - n_0\alpha_0-\sum_{i=1}^\ell\, n_i\alpha_i$, for some 
$n_i\in \bz_+$. Then, 
\[\lambda+\delta=\Lambda - (n_0-1)\alpha_0-\sum_{i=1}^\ell\, (n_i-h_i)\alpha_i.\]
Again applying  [Kac, Proposition 12.5(a)], $\lambda+\delta\notin P(\Lambda)$ if and only if $\lambda+\delta\not\leq \Lambda $,
i.e., for some $0\leq i\leq \ell$, $n_i< h_i$. Thus, $\lambda \in S(\Lambda)$. This proves the proposition.
\end{proof}

\section{$A_{1}^{(1)}$ Case}

In this section, we consider  $\mathfrak{g}=\widehat{\mathfrak{sl}_{2}}=\left(\bigoplus_{n\in\mathbb{Z}}\mathbb{C}t^{n}\otimes
\mathfrak{sl}_{2}\right)\oplus\mathbb{C}c\oplus\bc d$.
In this case $\mathfrak{h}^{*}=\mathbb{C}\alpha\oplus\mathbb{C}\delta\oplus\mathbb{C}\Lambda_{0}$, where 
$\alpha$ is the simple root of $\mathfrak{sl}_{2}$ and 
${\Lambda_0}_{|\overset{\circ}{\mathfrak{h}}\oplus \bc d}\equiv 0$ and $\Lambda_0(c)=1$. Then, $\Lambda_0$ is a zeroeth fundamental weight. 
The simple roots of $\widehat{\mathfrak{sl}_{2}}$ are 
 $\alpha_0:=\delta-\alpha$ and $\alpha_1:=\alpha$. The simple coroots are 
 $\alpha_0^\vee :=c-\alpha^\vee$ and $\alpha_1^\vee :=\alpha^\vee$. It is easy to see that an element of $\fh^*$ of the form $m\Lambda_0+\frac{j}{2}\alpha$ belongs to $P_+$ if and only if $m,j\in \bz_+$ and $m\geq j$.

Specializing Proposition \ref{prop4.1} to the case of $\mathfrak{g}=\widehat{\mathfrak{sl}_{2}}$, we get the following.

\begin{corollary} \label{cor5.1} For  $\mathfrak{g}=\widehat{\mathfrak{sl}_{2}}$
and  $\Lambda=m\Lambda_{0}+\frac{j}{2}\alpha\in P_+$, 
\beqn \label{e5.2}
P^o(\Lambda)_+=\left\{ \Lambda-k\alpha,\,\Lambda-l(\delta -\alpha)\;:\; k,l\in \bz_+ \,\,\text{and}\,\,  k\leq\frac{j}{2},\, l\leq\frac{m-j}{2}\right\}. 
\eeqn
\end{corollary}
\begin{proof} 
The corollary  follows from Proposition \ref{prop4.1} since $m_1\Lambda_0+\frac{m_2}{2}\alpha +
 m_3 \delta $ belongs to $P_+$ if and only if $m_1,m_2\in \bz_+$ and $m_1\geq m_2$.
\end{proof}

Let $\pi$ be the projection  $\mathfrak{h}^{*}=\mathbb{C}\Lambda_{0}\oplus \bc \alpha\oplus\mathbb{C}\delta
\to \mathbb{C}\Lambda_{0}\oplus \bc \alpha$. 
\begin{lemma} \label{lemma5.1} Let $\mathfrak{g}=\widehat{\mathfrak{sl}_{2}}$. For $\Lambda=m\Lambda_{0}+\frac{j}{2}\alpha \in P_+$ (i.e.,  
$m,j\in \bz_+$ and $m\geq j$) such that $m>0$,
\beqn\label{e5.1} \pi(P^o(\Lambda))=\{\Lambda+k\alpha: k\in \bz\}.
\eeqn
Moreover, for any $k\in \bz$, let $n_k$ be the unique integer such that  $\Lambda+k\alpha+n_k\delta\in P^o(\Lambda)$. Then, writing
$k=qm+r, 0\leq r<m$, we have:
\beqn\label{ne18}
n_k=
n_r-q(k+r+j).
\eeqn
\end{lemma}
\begin{proof} The assertion \eqref{e5.1} follows from the identity \eqref{e5.2} together with the action of the affine Weyl group 
$W\simeq \overset{\circ}{W}\times (\bz\alpha^\vee)$ on $\fh^*$, where $\overset{\circ}{W}$ is the Weyl group of $\mathfrak{sl}_{2}$
and $\bz\alpha^\vee$ acts on $\fh^*$ via:
\beqn\label{e5.3} T_{n\alpha^\vee}(\mu)=\mu+ n\mu(c)\alpha-[n\mu(\alpha^\vee)+n^2\mu(c)]\delta,\,\,\,\text{for}\, n\in \bz, \mu\in \fh^*.
\eeqn
Since $P^o(\Lambda)$ is $W$-stable, the identity \eqref{ne18} can be established from the action of the affine Weyl group 
element $T_{-q\alpha^\vee}$ on $\Lambda+k\alpha+n_k\delta$. 
\end{proof}
The value of $n_r$ for $0\leq r<m$ can be determined  from the identity \eqref{e5.2}
by applying $T_{\alpha^\vee}, T_{\alpha^\vee}\cdot s_1$ to $\Lambda-k\alpha$ and applying 
$1, T_{\alpha^\vee}\cdot s_1$ to $\Lambda-l(\delta-\alpha)$, where $s_1$ is the nontrivial element of $\overset{\circ}{W}$. We record 
the result in the following lemma.
\begin{lemma}\label{lemma5.2'} With the notation as in the above lemma, the value of $n_r$ for any integer $0\leq r<m$ is given by 
\[
n_r=\begin{cases}
-r\,,& \text{for}\; 0\leq r\leq m-j\\
m-j-2r & \text{for}\; m-j\leq r <m.
\end{cases}\]
\end{lemma}
\begin{lemma} \label{lemma5.1'} Take the following elements in $P_+$:
\[\Lambda=m\Lambda_{0}+\frac{j}{2}\alpha,\;\Lambda'=m'\Lambda_{0}+\frac{j'}{2}\alpha,\;\Lambda''=m''\Lambda_{0}+\frac{j''}{2}\alpha,
\]
where $m:=m'+m''$ and we assume that $m'>0$. 
Then,
\begin{align*}
\pi\left(T_{\Lambda}^{\Lambda',\,\Lambda''}\right)=\{ \Lambda'+k\alpha\;:\; k\in\mathbb{Z},\, & k\equiv\frac{1}{2}\left(j-j'-j''\right) \\
&\text{or}\,\,  k\equiv -\frac{1}{2}\left(j+j'+j''\right)-1\,\text{mod}\,M\} ,
\end{align*}
where $M:=m+2$.  In particular, by the equation \eqref{eq1001}, $T_{\Lambda}^{\Lambda',\,\Lambda''}$ is nonempty if and
only if $\frac{j-j'-j''}{2}\in \bz$. 

Moreover, for 
$\lambda=\Lambda'+k\alpha+n_k\delta \in T_{\Lambda}^{\Lambda',\Lambda''}$,
\[
v_{\Lambda,\Lambda'',\,\lambda}=\begin{cases}
T_{\frac{k-\frac{1}{2}\left(j-j'-j''\right)}{M}\alpha^\vee}, & \text{if}\; k\equiv\frac{1}{2}\left(j-j'-j''\right)\,\mod\, M\\
s_{1}T_{-\frac{k+\frac{1}{2}\left(j+j'+j''\right)+1}{M}\alpha^\vee },& \text{if}\; k\equiv-\frac{1}{2}\left(j+j'+j''\right)-1\,\text{mod}\, M,
\end{cases}
\]
where $T_{n\alpha^\vee}$ is defined by the equation \eqref{e5.3}. Further,
\[
 S_{\Lambda,\Lambda'',\lambda}=n_k+\frac{\left(k-\frac{1}{2}\left(j-j'-j''\right)\right)\left(k+\frac{1}{2}\left(j+j'+j''\right)+1\right)}{M}.
\]
\end{lemma}
\begin{proof}
Follows from the fact that $W=\stackrel{\circ}{W}\rtimes\mathbb{Z}\alpha^\vee$
and that $\rho=2\Lambda_{0}+\frac{1}{2}\alpha$.\end{proof}
We have the following very crucial result.
\begin{prop}\label{maximum}
Fix $\Lambda, \Lambda'$ and $\Lambda''$ as in Lemma \ref{lemma5.1'} and asume that $\frac{j-j'-j''}{2}\in \bz$ and both of $m',m''>0$. Then, the maximum of 
$\left\{ S_{\Lambda,\Lambda'',\lambda}:\;\lambda\in T_{\Lambda}^{\Lambda',\,\Lambda''}\,\,\text{and}\, \,\varepsilon(v_{\Lambda,\Lambda'',\lambda})=1\right\} $
 is achieved precisely when $\pi(\lambda)=\Lambda'+\frac{1}{2}\left(j-j'-j''\right)\alpha$. 
\end{prop}
\begin{proof} By Lemma \ref{lemma5.1'} and the explicit description of the length function of $T_{n\alpha^\vee}$ (cf. [K$_3$, Exercise 13.1.E.3]), 
$$\pi\{\lambda\in T_{\Lambda}^{\Lambda',\,\Lambda''}:\varepsilon (v_{\Lambda,\Lambda'',\lambda})=1\}=\{\Lambda'+k_l\alpha:l\in \bz\},$$
where $M:=m+2$ and $k_l:=\frac{j-j'-j''}{2}+lM$. 
Take  $\lambda=\Lambda'+k_l\alpha\in \pi(T_{\Lambda}^{\Lambda',\,\Lambda''})$ for $l\in \bz$.
 Write $k_l=q_lm'+r_l$ for $q_l\in \bz$ and $0\leq r_l< m'$. Then, by Lemmas \ref{lemma5.1}, \ref{lemma5.2'} and \ref{lemma5.1'}, for $\lambda=\Lambda'+k_l\alpha$ (setting $J:=\frac{j-j'-j''}{2}$), 
\begin{align*}
S_{\Lambda,\Lambda'',\lambda} & =  n_{r_l}-\frac{(J+j'+lM+r_l)(J+lM-r_l)}{m'}+l(lM+1+j)\\
&=l^2M(1-\frac{M}{m'})+l(1+j-\frac{M(j-j'')}{m'})-\frac{(j-j'')^2-j'^2}{4m'}+\frac{r_l^2}{m'}
+\frac{r_lj'}{m'}+n_{r_l}\\
&= l^2M(1-\frac{M}{m'})+l(1+j-\frac{M}{m'}(j-j''))-\frac{(j-j'')^2-j'^2}{4m'}+p(k_l),
\end{align*}
where 
\[
p(k_l):= \frac{r_l^2}{m'}+\frac{r_l}{m'}j' +n_{k_l}.
\]
Let $P=P_{m',j'}:\mathbb{R}\rightarrow\mathbb{R}$ be the following function:
\begin{eqnarray*}
 P(s) & := & \begin{cases}
\frac{(s-\frac{m'}{2}k)^{2}}{m'}-\frac{(j')^2}{4m'}, & \text{if}\;\left|s-\frac{m'}{2}k\right|\leq\frac{j'}{2}\;\text{for some }k\in2\mathbb{Z}\\
\frac{(s-\frac{m'}{2}k)^{2}}{m'}-\frac{(m'-j')^2}{4m'}, & \text{if}\;\left|s-\frac{m'}{2}k\right|\leq\frac{m'-j'}{2}\;\text{for some }k\in2\mathbb{Z}+1.
\end{cases}
\end{eqnarray*}
Let $k_s\in \bz$ be such a $k$. (Of course, $k_s$ depends upon $m'$ and $j'$.)
\begin{claim}
$P(s) = p (s-\frac{j'}{2})$ for $s \in \frac{j'}{2}+\mathbb{Z}$.
\end{claim}
\begin{proof}
Clearly, both of $P$ and $p$ are  periodic with period $m'$. So, it is enough to show that $P(s) = p (s-\frac{j'}{2})$, for $s-\frac{j'}{2}$
equal to any of the integral points of the interval $[-j', m'-j']$. By Lemma \ref{lemma5.2'} and the identity \eqref{ne18},  for
any integer $-j'\leq r\leq 0$,
$$p (r)= \frac{1}{m'}r(r+j'),$$
and for any integer $0\leq r\leq m'-j'$, 
$$p (r)= \frac{r(r+j')}{m'}-r.$$
From this, the claim follows immediately. 
\end{proof}
 Fix $m'>0$. Let 
\begin{align*}
I :=\{(t,j',m'',j'',j)\in\mathbb{R}^5 \,:\, &0\leq j' \leq m',\, 1 \leq m'',\\
 &0\leq j'' \leq m'',\, 0\leq j\leq m'+m''\}.
\end{align*}
Define $F: I \rightarrow \mathbb{R}$ by
\begin{eqnarray*}
F: (t,j',m'',j'',j) &\mapsto&   t^{2}M(1-\frac{M}{m'})+t\bigl(j(1-\frac{M}{m'})+1+\frac{M}{m'}j''\bigr)\\
&&+\frac{(j')^2-(j-j'')^{2}}{4m'} +P(\frac{1}{2}\left(j-j''\right)+tM).
\end{eqnarray*}
Thus, $F$ is a continuous, piecewise smooth function with failure of differentiability along the set $$\{(t,j',m'',j'',j)\in I\,:\, \frac{1}{2}(j\pm j'-j'')+tM\in m'\mathbb{Z} \}.$$
\begin{claim}\label{claim1}
Let $\Delta (t)=\Delta(t, j',m'',j'', j):=F(t+1,j',m'',j'',j)-F(t,j',m'',j'',j)$. Then, on $I$,
\begin{enumerate}
\item $\Delta$ is a nonincreasing function of $t$
\item $\Delta$ is increasing with respect to $j''$
\item $\Delta$ is nonincreasing in $j$
\item\label{delclaimb} $\Delta(0)$ is decreasing in $m''$
\item\label{delclaima} $\Delta(-1)$ is nondecreasing in $m''$.\end{enumerate}
\end{claim}

\begin{proof}
We compute and give bounds for the partial derivatives of $\Delta$, where they exist. 
\begin{eqnarray*}
\Delta(t) & =&\,2tM(1-\frac{M}{m'})+\bigl((j+M)(1-\frac{M}{m'})+1+\frac{M}{m'}j''\bigr)\\
&&+P(tM+M+\frac{1}{2}(j-j''))-P(tM+\frac{1}{2}(j-j'')).
\end{eqnarray*}
Hence, 
\begin{eqnarray*}
\partial_t\Delta(t) & = & 2M(1-\frac{M}{m'})+M\bigl(P'(tM+M+\frac{1}{2}(j-j''))-P'(tM+\frac{1}{2}(j-j''))\bigr)\\
 & = & 2M(1-\frac{M}{m'})+2\frac{M}{m'}(M-\frac{m'}{2}k_{1}+\frac{m'}{2}k_{0})\\
 & = & 2M(1-\frac{k_{1}-k_{0}}{2}),
\end{eqnarray*}
where $k_1:=k_{(t+1)M+\frac{1}{2}(j-j'')}$ and $k_0:=k_{tM+\frac{1}{2}(j-j'')}$.
Since $2\leq k_{1}-k_{0}$, we see that $\partial_t\Delta\leq0$, wherever $\partial_t \Delta$ exists. 
 Since $\Delta$ is continuous everywhere and differentiable 
on all but a discrete set, $\Delta$ is nonincreasing in $t$.
\[
\partial_{j''}\Delta(t)=\frac{M}{m'}-\frac{1}{2}\left(P'(tM+M+\frac{1}{2}(j-j''))-P'(tM+\frac{1}{2}(j-j''))\right).
\]
Now, $|P'|\leq1$, so $\frac{M}{m'}+1\geq\partial_{j''}\Delta\geq\frac{M}{m'}-1=\frac{m''+2}{m'}>0$.

For (3):
\begin{eqnarray*}
\partial_{j}\Delta(t) & = & 1-\frac{M}{m'}+\frac{1}{2}\bigl(P'(tM+M+\frac{1}{2}(j-j''))-P'(tM+\frac{1}{2}(j-j''))\bigr)\\
 & = & 1-\frac{M}{m'}+\frac{1}{m'}\left(M-\frac{m'}{2}k_{1}+\frac{m'}{2}k_{0}\right)\\
 & = & 1-\frac{k_{1}-k_{0}}{2}\leq 0.
\end{eqnarray*}
  (\ref{delclaimb}) and (\ref{delclaima}) follow from the following calculation:
\begin{align*}
\partial_{m''}\Delta=&2t(1-2\frac{M}{m'})+(1-2\frac{M}{m'}+\frac{1}{m'}(j''-j))\\
&+(t+1)P'(tM+M+\frac{1}{2}(j-j''))-tP'(tM+\frac{1}{2}(j-j'')).
\end{align*}
Hence, 
\begin{eqnarray*}
\partial_{m''}\Delta(0) & = & 1-2\frac{M}{m'}+\frac{1}{m'}(j''-j)+P'(M+\frac{1}{2}(j-j''))\\
 & \leq & 1-2\frac{M}{m'}+\frac{m''}{m'}+1\\
 & = & \frac{-m''-4}{m'}<0,
\end{eqnarray*}
and 
\begin{eqnarray*}
\partial_{m''}\Delta(-1) & = & -2(1-2\frac{M}{m'})+(1-2\frac{M}{m'}+\frac{1}{m'}(j''-j))+P'(-M+\frac{1}{2}(j-j''))\\
 & = & -1+2\frac{M}{m'}+\frac{1}{m'}(j''-j)+P'(-M+\frac{1}{2}(j-j''))\\
 & = & -1+2\frac{M}{m'}+\frac{1}{m'}(j''-j)-2\frac{M}{m'}+\frac{1}{m'}(j-j'')-k_{0}\\
 & = & -1-k_{0}.
\end{eqnarray*}
Note that $k_{0}\leq-1$ since $-\frac{(j-j'')}{2}-M<-\frac{m'}{2}$.
Thus, $\partial_{m''}\Delta(-1)\geq0$.\end{proof}
\begin{claim}
{\it The maximum of $F=F(-, j',m'',j'',j):\,\mathbb{Z}\rightarrow\mathbb{R}$ occurs at $0$.}
\end{claim}
\begin{proof}
We show that $\Delta(-1)>0 > \Delta(0)$. Since $\Delta$ is
nonincreasing in $t$, it would follow that $F(0)>F(t)$ for all $t\in\mathbb{Z}_{\neq0}$.

Let us begin with $\Delta(-1)$. By the previous claim \ref{claim1}, $\Delta(-1)$
is as small as possible when $m''=1$, $j''=0$, and $j=m'+1$. So,
let us compute with these values:

\begin{eqnarray*}
\Delta(-1) & \geq & \frac{6}{m'}+1+P(\frac{1}{2}m'+\frac{1}{2})-P(-2-\frac{1}{2}m'-\frac{1}{2})\\
 & = & \frac{6}{m'}+1+\frac{(\frac{1}{2}m'+\frac{1}{2}-\frac{1}{2}m'k_{1})^{2}}{m'}-\frac{(2+\frac{1}{2}m'+\frac{1}{2}+
\frac{1}{2}m'k_{0})^{2}}{m'}\\
&+&\begin{cases}
\frac{m'}{4}-\frac{j'}{2} & \text{ if }k_{0}\text{ odd},\, k_{1}\text{ even}\\
0 & \text{ if }k_{1}-k_{0}\text{ even}\\
\frac{j'}{2}-\frac{m'}{4} & \text{ if }k_{1}\text{ odd},\, k_{0}\text{ even}.
\end{cases}
\end{eqnarray*}
Note that for $m'\geq 5$, the possible values of $(k_{1},k_{0})$ are
$(1,-1)$; $(1,-2)$; or $(2,-2)$.   So,  the result, that $\Delta (-1)>0$,  is established
by considering such pairs directly and by cases for smaller $m'$.

For $\Delta(0)$, we take $m''=1$, $j''=1$, and $j=0$.
\begin{eqnarray*}
\Delta(0) & = & \bigl(\frac{-3(3+m')}{m'}+1+\frac{3+m'}{m'}\bigr)+P(\frac{1}{2}+2+m')-P(-\frac{1}{2})\\
 & =&1 - \frac{2(3+m')}{m'}+P(\frac{1}{2}+2+m')-P(-\frac{1}{2})\\
 & = & 1-\frac{2(3+m')}{m'}+\frac{(\frac{1}{2}+2+m'-\frac{1}{2}m'k_{1})^{2}}{m'}-\frac{(\frac{1}{2}+\frac{1}{2}m'k_{0})^{2}}{m'}
\\ &+&\begin{cases}
\frac{m'}{4}-\frac{j'}{2} & \text{ if }k_{0}\text{ odd},\, k_{1}\text{ even}\\
0 & \text{ if }k_{1}-k_{0}\text{ even}\\
\frac{j'}{2}-\frac{m'}{4} & \text{ if }k_{1}\text{ odd},\, k_{0}\text{ even}.
\end{cases}
\end{eqnarray*}
For $m'\geq 5$,  the possible values of $(k_{1}, k_{0})$ are
 $(3,-1)$; $(3,0)$; or $(2,0)$. So, again the result, that $\Delta (0)<0$, is established
by considering such pairs directly and by cases for smaller $m'$.
\end{proof}
 This completes the proof of the proposition.
\end{proof}
\begin{rem}\label{newremark}
We have shown that $F(l,j',m'',j'',j) = S_{\Lambda,\Lambda'',\lambda}$ for integral values of $l$. If $l$ is not an integer, then $\lambda_l := \Lambda' + (lM+J)\alpha$ may not be in $\pi(T_\Lambda^{\Lambda',\Lambda''})$, in which case $ S_{\Lambda,\Lambda'',\lambda_l}$ is not defined. On the other hand, if $\lambda_l \in \pi(T_\Lambda^{\Lambda',\Lambda''})$, we note that the equality  $F(l,j',m'',j'',j) = S_{\Lambda,\Lambda'',\lambda_l}$  holds, as can be seen by letting $k_l = lM -\frac{1}{2} (j+j'+j'')-1$ in the above proof.
\end{rem}

Now, let us apply the same analysis to the case that $\varepsilon(v_{\Lambda,\Lambda'',\lambda})=-1$.
By Lemma \ref{lemma5.1'}, this corresponds to $k_l=-\frac{1}{2}\left(j+j'+j''\right)-1+lM$.
 For 
$\lambda=\Lambda'+k_l\alpha$, let us denote the function $S_{\Lambda,\Lambda'',\lambda}$ by $G_\bz(l)=G_\bz(l, j',m'',j'',j)$. Thus,
$G_\bz:\bz \to \bz$. 
\begin{lemma} \label{lemma5.9} Define the function $G=G(-, j',m'',j'',j):\br \to \br$ by 
\[G(t, j',m'',j'',j)=F(t-\frac{j+1}{M}, j',m'',j'',j).\]
Then, $G_{|\bz}=G_\bz.$

Hence, $S_{\Lambda,\Lambda'',\lambda}$
has a maximum when $l=0$ or $l=1$.
\end{lemma}
\begin{proof}
By the proof of Proposition \ref{maximum} and Remark \ref{newremark}, $S_{\Lambda,\Lambda'',\lambda+(j+1)\alpha} = F(l)$, for $\lambda=\Lambda'+k_l\alpha$. Since $\lambda = \Lambda' +(-\frac{1}{2}\left(j+j'+j''\right)-1+lM)\alpha$, by Proposition \ref{maximum}, $S_{\Lambda,\Lambda'',\lambda} = F(l-\frac{j+1}{M})$.
 This proves the lemma.
\end{proof}

From Lmma \ref{lemma5.9} and the definition of $F$, it is easy to see that
\begin{equation}\label{eq22}
G(1-t,m'-j',m'',m''-j'',m'+m''-j)+\frac{1}{2}(j'+j''-j)=G(t, j',m'', j'',j),
\end{equation}
for any $t\in \br$.
Hence, if the maximum of $G_\bz$ occurs at
1, it is equal to 
\begin{equation}\label{eq23}
G(0, m'-j',m'',m''-j'',m'+m''-j)+\frac{1}{2}(j'+j''-j).
\end{equation}
We also record the following identity, which is easy to prove from the definition of $F$.
\begin{equation}\label{neq23}
F (0,m'-j',m'', m''-j'',m'+m''-j)+\frac{1}{2}(j'+j''-j)=F(0,j',m'', j'',j).
\end{equation}

As a  corollary of Proposition \ref{maximum} and Lemma \ref{lemma5.9}, we get the following `Non-Cancellation Lemma'.
\begin{corollary}\label{corcancel} Let $\Lambda, \Lambda',\Lambda''$ be  as in Proposition \ref{maximum} and let
\begin{align*}
\mu^{\Lambda',\Lambda''}_{\Lambda}&:=\max\left\{ S_{\Lambda,\Lambda'',\lambda}:\;\lambda\in T_{\Lambda}^{\Lambda',\Lambda''}\,\,\,\text{and}\,\,\varepsilon(v_{\Lambda,\Lambda'',\lambda})=1\right\},\\
\bar{\mu}^{\Lambda',\Lambda''}_{\Lambda}&:=\max\left\{ S_{\Lambda,\Lambda'',\lambda}  :\;\lambda\in T_{\Lambda}^{\Lambda',\Lambda''}\,\,\,\text{and}\,\,\varepsilon(v_{\Lambda,\Lambda'',\lambda})=
-1\right\}. 
\end{align*}
Assume that ${\mu}^{\Lambda',\Lambda''}_{\Lambda}=\bar{\mu}^{\Lambda',\Lambda''}_{\Lambda}$. Then,
\[{\mu}^{\Lambda'',\Lambda'}_{\Lambda}\neq\bar{\mu}^{\Lambda'',\Lambda'}_{\Lambda}.\]
\end{corollary}
\begin{proof}
We proceed in two cases:

{\em Case I.}  Suppose the maximum $\bar{\mu}^{\Lambda',\Lambda''}_{\Lambda}$ 
occurs when $\pi(\lambda)=\Lambda'-(\frac{1}{2}\left(j+j'+j''\right)+1)\alpha$ (cf. Lemma \ref{lemma5.9}).
This means that the $\delta$-maximal weights of $L(\Lambda')$ through
$\Lambda'-(\frac{1}{2}\left(j+j'+j''\right)+1)\alpha$ and through
$\Lambda'+\frac{1}{2}\left(j-j'-j''\right)\alpha$ have the same $\delta$
coordinate (cf. Proposition \ref{maximum}). By (next) Lemma \ref{lemma5.11}, 
we know that this occurs if and only if 
one of the following two conditions are satisfied:

(1) $\left|\frac{1}{2}\left(j-j''\right)\right|\leq\frac{j'}{2}$
and $\frac{1}{2}\left(j+j''\right)+1\leq\frac{j'}{2}$, or 

(2) $\frac{1}{2}\left(j+j''\right)+1=\frac{1}{2}\left(j-j''\right)$.

The latter is clearly impossible, while the former condition is fulfilled
precisely when $\frac{1}{2}\left(j+j''\right)+1\leq\frac{j'}{2}$.

 So,  for the equality ${\mu}^{\Lambda',\Lambda''}_{\Lambda}=\bar{\mu}^{\Lambda',\Lambda''}_{\Lambda}$ in this case,
the neccesary and sufficient condition is: 
\beqn \label{e5.10.3} \frac{1}{2}\left(j+j''\right)+1\leq\frac{j'}{2}.
\eeqn

{\em Case II.} Suppose the maximum $\bar{\mu}^{\Lambda',\Lambda''}_{\Lambda}$ 
occurs when $\pi(\lambda )= \Lambda'-(\frac{1}{2}\left(j+j'+j''\right)+1-M)\alpha$.
Then, by the identities \eqref{eq23} and \eqref{neq23}, we get 
\beqn\label{eq24} G (0,m'-j',m'', m''-j'',m'+m''-j)=F (0,m'-j',m'', m''-j'',m'+m''-j).
\eeqn
So, from the case I, we get in this case II, 
${\mu}^{\Lambda',\Lambda''}_{\Lambda}=\bar{\mu}^{\Lambda',\Lambda''}_{\Lambda}$ if and only if
\beqn\label{eq25}
\frac{1}{2}\bigl((m'+m''-j)+(m''-j'')\bigr)+1\leq\frac{1}{2}(m'-j').
\eeqn
So, if either of the inequalities \eqref{e5.10.3} or \eqref{eq25} is satisfied,  then none of them can be satisfied 
for the triple $(\Lambda, \Lambda',\Lambda'')$ replaced by $(\Lambda, \Lambda'',\Lambda')$. This proves the corollary.
\end{proof}
\begin{lem}\label{lemma5.11}
Suppose $\Lambda'-(\frac{1}{2}\left(j+j'+j''\right)+1)\alpha+n_1 \delta$ and $\Lambda'+\frac{1}{2}\left(j-j'-j''\right)\alpha+ n_2 \delta$ are $\delta$-maximal weights of $L(\Lambda')$. Then $n_1 = n_2$ if and only if 
\[
\left|\frac{1}{2}\left(j-j''\right)\right|\leq\frac{j'}{2} \quad \text{and}\,\,\,\,
\frac{1}{2}\left(j+j''\right)+1\leq\frac{j'}{2},
\]
or $\frac{1}{2}\left(j+j''\right)+1=\frac{1}{2}\left(j-j''\right)$.
\end{lem}
\begin{proof}
Fix an integer $n$ and consider the set $P_n = \{ \nu \in P(\Lambda') \, : \, \Lambda' - \nu = k\alpha + n \delta,\, k \in \bz \}$. We give a description of $P_n \cap P^o(\Lambda')$. Clearly, $P_n = \{\lambda, \lambda -\alpha, \dots, \lambda - \langle\lambda,\alpha^\vee\rangle\alpha \}$ for some $\lambda=\lambda_n$ and that this $\lambda$ is uniquely determined by $n$ (cf. [K$_3$, Exercise 2.3.E.2]). Suppose that some $\mu \in P_n$ is not $\delta$-maximal, then none of $\{\mu, \dots, \mu - \langle\mu, \alpha^\vee\rangle\alpha \}$ are $\delta$-maximal, since
if $\mu+k\delta\in P(\Lambda')$, then the whole string $\{\mu +k\delta, \dots, \mu +k\delta- \langle\mu, \alpha^\vee\rangle\alpha \}
\subset P(\Lambda')$. In particular, if $\lambda - \alpha$ is $\delta$-maximal, then so is $\lambda$. Hence, $\mathfrak{g}_{\delta-\alpha}L(\Lambda')_\lambda = 0$ and  $\mathfrak{g}_{\alpha}L(\Lambda')_\lambda = 0$. Therefore, $\lambda$ is the highest weight $\Lambda'$. Thus, $P_n \cap P^o(\Lambda')$ is either empty, or $\lambda=\Lambda'$   (in the case that $n = 0$), or the set $\{\lambda,s_1\lambda\}$. From this and Corollary \ref{cor5.1} the lemma follows easily.
\end{proof}

\section{Saturation factor for the $A_{1}^{(1)}$ Case}

 We assume that $\mathfrak{g}=\widehat{\mathfrak{sl}_{2}}$ in this section.
\begin{defn}
Let $\Lambda'\in P_{+}^{(m')}, \Lambda''\in P_{+}^{(m'')}$ and $\Lambda\in P_{+}^{(m'+m'')}$. Then, we call $L(\Lambda+n\delta)$
the \textit{$\delta$-maximal component of $L(\Lambda')\otimes L(\Lambda'')$
through $\Lambda$} if $L(\Lambda+n\delta)$ is a submodule of $L(\Lambda')\otimes L(\Lambda'')$
but $L(\Lambda+m\delta)$ is not a component for any $m>n$.\end{defn}
\begin{thm}  \label{thm6.2} Let $\Lambda', \Lambda'', \Lambda$ be as in Proposition \ref{maximum} . Then, 
$L(\Lambda+n\delta)$ is a $\delta$-maximal component of $L(\Lambda')\otimes L(\Lambda'')$
if $n=\min(n_{1},n_{2})$, where $n_{1}$ is such that $\Lambda-\Lambda''+n_{1}\delta\in P^o(\Lambda')$
and $n_{2}$ is such that $\Lambda-\Lambda'+n_{2}\delta\in P^o(\Lambda'')$.\end{thm}
\begin{proof} This follows immediately by combining Propositions \ref{tensor}, \ref{maximum} and Lemma \ref{lemma5.1'}.
\end{proof}
\begin{lem} Fix a positive integer $N$. 
Let $\Lambda\in \bar{P}_+$ and let $\lambda\in\Lambda+Q$, where $Q$ is the root lattice $\mathbb{Z}\alpha\oplus\mathbb{Z}\delta$ of
$\widehat{\mathfrak{sl}_{2}}$. 
Then, $N\lambda\in P^o(N\Lambda)$ if
and only if $\lambda\in P^o(\Lambda)$. 
\end{lem}
\begin{proof}
The validity of the lemma  is clear for $\lambda\in P^o(\Lambda)_+$ from Corollary \ref{cor5.1}.  But since
$ P^o(\Lambda)=W\cdot ( P^o(\Lambda)_+)$,
 and the action of $W$ on $\fh^*$ is linear,  the lemma follows for any $\lambda \in P^o(\Lambda)$.\end{proof}
\begin{cor}\label{cor6.4}
Let $d_o\in\mathbb{Z}_{>1}$. Let $\Lambda$, $\Lambda'$, $\Lambda'' \in P_+$
be such that  $\Lambda-\Lambda'-\Lambda''\in Q$ and 
 $L(N\Lambda)$ is a submodule of $L(N\Lambda')\otimes L(N\Lambda'')$,  for some $N\in \bz_{>0}$.  
Then,  $L(d_o\Lambda)$ is a submodule of $L(d_o\Lambda')\otimes L(d_o\Lambda'')$.

Such a $d_o$ is called  a {\em saturation factor}. 
\end{cor}
\begin{proof} If $\Lambda'(c)=0$ or $\Lambda''(c)=0$, then 
$$L(N\Lambda')\otimes L(N\Lambda'')\simeq L(N(\Lambda'+\Lambda'')),$$
for any $N\geq 1$. Thus, the corollary is clearly true in this case. So, let us assume that both of 
$\Lambda'(c)>0$ and $\Lambda''(c)>0$. 
Let $L(N\Lambda+n\delta)$ be the $\delta$-maximal component of $L(N\Lambda')\otimes L(N\Lambda'')$
through $L(N\Lambda)$, for some $n\geq 0$. For any $\Psi\in P_+$, let $\bar{\Psi}\in \bar{P}_+$  be the projection $\pi(\Psi)$ defined just before Lemma \ref{lemma5.1}. Applying Theorem \ref{thm6.2} to $\bar{\Lambda}', \bar{\Lambda}'', \bar{\Lambda}$, and observing that 
\beqn \label{eq6.4.0} L(\bar{\Psi}+k\delta)\simeq L(\bar{\Psi})\otimes L(k\delta)\eeqn
and $L(k\delta)$ is one dimensional, we get that there is a $\delta$-maximal component $L(\Lambda+\tilde{n}\delta)$ of 
$L(\Lambda')\otimes L(\Lambda'')$ through $L(\Lambda)$, for some (unique) $\tilde{n}\in \bz$. 

Again applying Theorem \ref{thm6.2} to $N\bar{\Lambda}', N\bar{\Lambda}'', N\bar{\Lambda}$, and observing 
(using Corollary \ref{cor5.1}) that 
\beqn \label{eq6.4.1} P^o(N\bar{\Psi})\supset NP^o(\bar{\Psi}),\eeqn
we get that $L(N\Lambda+N\tilde{n}\delta)$ is the $\delta$-maximal component  of 
$L(N\Lambda')\otimes L(N\Lambda'')$ through $L(N\Lambda)$. Thus, $n=N\tilde{n}$. In particular,
\beqn\label{e6.3} \tilde{n}\geq 0.\eeqn
 Let 
\beqn\label{e6.4}
\sum_{\lambda\in T_{\bar{\Lambda}}^{\Lambda',\Lambda''}}\varepsilon(v_{\bar{\Lambda},\Lambda'',\lambda})c_{\Lambda',\lambda}e^{S_{\bar{\Lambda},\Lambda'',\lambda}\delta}=\sum_{k\in\mathbb{Z}_+}c_{k}e^{(\Lambda(d)+\tilde{n}-k)\delta},
\eeqn
for some $c_k\in \bz_+$ with $c_0$ nonzero. 
By Proposition \ref{tensor}, this is the character of a unitarizable 
Virasoro representation with each irreducible component having the same nonzero central charge. Thus, by Lemma 
\ref{virasoro}, for any $k>1$, we get 
$c_{k}\neq0$. 

By the above argument,  $L(d_o\Lambda+d_o\tilde{n}\delta)$ is the $\delta$-maximal component  of 
$L(d_o\Lambda')\otimes L(d_o\Lambda'')$ through $L(d_o\Lambda)$. If $\tilde{n}=0$, we get that 
 $$L(d_o\Lambda) \subset L(d_o\Lambda')\otimes L(d_o\Lambda'').$$
If $\tilde{n}>0$, then $d_o\tilde{n}$ being $>1$, by the analogue of  \eqref{e6.4} for 
$d_o\Lambda', d_o\Lambda''$ and $d_o\Lambda$,
$L(d_o\Lambda) \subset L(d_o\Lambda')\otimes L(d_o\Lambda'').$ (Here we have used that 
$L_0=-d+p$ on any $\fg$-isotypical component of $L(\Lambda')\otimes L(\Lambda'')$ with highest weight in $\Lambda+\bz \delta$, for 
a number $p$ depending only upon $\bar{\Lambda}, \Lambda'$ and $\Lambda''$, cf. [KR, Identity 10.25 on page 116].)
This proves the corollary.
\end{proof}

\begin{rem}
We note that $L(2\Lambda_{0}-\delta)$ is not a component of $L(\Lambda_{0})\otimes L(\Lambda_{0})$
(cf. [Kac, Exercise 12.16]).
But, of course,  $L(2\Lambda_{0})$ is a $\delta$-maximal component. By the identity \eqref{e6.4}, we know that $L(2d_o\Lambda_{0}-d_o\delta)$
must be a component of $L(d_o\Lambda_{0})\otimes L(d_o\Lambda_{0})$, for any $d_o>1$. 
So $d_o$ can not be taken to be $1$ in Corollary \ref{cor6.4}.
\end{rem}

\section{A Conjecture}

In this section, $G$ is any symmetrizable Kac-Moody group. 
We recall the following definition of the deformed product due to Belkale-Kumar [BK]. (Even though they gave the definition in the finite case, the same definition works in the symmetrizable Kac-Moody case, though with only one parameter.)
\begin{definition}
Let $P$ be any standard parabolic subgroup of $G$. Recall from Section 2 that  $\{\epsilon^w_P\}_{w\in W^P}$ is a
basis of the singular cohomology $H^*(X_P, \bz)$. 
Write the standard cup product in $H^*(X_P, \Bbb Z)$ in this basis as follows:
\begin{equation}\label{constants}
\epsilon_P^u\cdot \epsilon_P^v=\sum_{w\in W^P} n^w_{u,v}\epsilon_P^w,\,\,\,\text{for some (unique)}\,  n^w_{u,v}\in \bz.
\end{equation}
Introduce the indeterminate ${\tau}$  and define a deformed cup product $\odot$
as follows:
\begin{equation}\label{10n}
\epsilon_P^u \odot \epsilon_P^v=
\sum_{w\in W^P} 
{\tau}^{(u^{-1}\rho+v^{-1}\rho -w^{-1}\rho -\rho)(x_P)}
n^w_{u,v} \epsilon_P^w,
\end{equation}
where $x_P:=\sum_{\alpha_i\in
\Delta\setminus\Delta(P)}\,x_i$, $\Delta(P)$ is the set of simple roots of the Levi $L$ of $P$ and, as in Section 2, $\Delta$ is 
the set of simple roots of $G$. 

The following lemma is a generalization of the corresponding result in the finite case (cf. [BK, Proposition 17]). 
\begin{proposition} (a) The product $\odot$ is associative and clearly commutative. 

(b) 
Whenever $n^w_{u,v}$ is nonzero,
the exponent of $\tau$ in the above is a nonnegative integer. 
\end{proposition}
\begin{proof} The proof of the associativity of $\odot$ is identical to the proof given in [BK, Proof of Proposition 17 (b)]. 

(b)  The proof of this part follows the proof of [BK, Theorem 43]. Consider the
decreasing filtration $\ca =\{\ca_m\}_{m\geq 0}$ of
$H^*(X_P, \bc )$ defined as follows:
  \[
\ca_m := \bigoplus_{w\in W^P:(\rho-w^{-1}\rho) (x_P)\geq m} \bc \epsilon_P^w.
  \]
A priori $\{\ca_m\}_{m\geq 0}$ may not be a multiplicative filtration. 

We next introduce another filtration $\{\bar\cf_m\}_{m\geq 0}$ of
$H^*(X_P,\bc )$ in terms of the Lie algebra cohomology.  Recall
that  $H^*(X_P,\bc )$ can be identified canonically with the Lie
algebra cohomology $H^*(\fg ,\fl )$, where $\fl$ is the Lie algebra of the Levi subgroup
$L$ of $P$ (cf. [K$_2$, Theorem 1.6]).  The underlying cochain
complex $C\u. =C\u. (\fg ,\fl )$ for $H^*(\fg ,\fl )$ can be
written as
  \[
C\u. := [\wed\u. (\fg /\fl )^*]^{\fl} = \Hom_{\fl} \bigl( \wed\u.
(\fu_P)\otimes\wed\u. (\fu_P^-), \bc \bigr),
  \]
where $\fu_P$ (resp. $\fu_P^-$) is the nil-radical of the Lie algebra of $P$ (resp. the opposite parabolic subgroup $P^-$). 
Define a decreasing multiplicative filtration $\cf =\{\cf_m\}_{m\geq 0}$ of the
cochain
complex $C\u.$ by subcomplexes:
  \[
\cf_m := \Hom_{\fl}\Biggl( \frac{\wed\u. (\fu_P)\otimes\wed\u.
(\fu^-_P)}{\bigoplus_{s+t\leq m-1}
\wed\u._{(s)}(\fu_P)\otimes\wed\u._{(t)}(\fu^-_P)}, \bc\Biggr) ,
  \]
where $\wed\u._{(s)}(\fu_P)$ (resp. $\wed\u._{(s)}(\fu^-_P)$) denotes the
subspace of $\wed\u. (\fu_P)$ (resp. $\wed\u. (\fu^-_P)$) spanned by the
$\fh$-weight vectors of weight $\beta$ with {\em $P$-relative height}
$$\text{ht}_P (\beta ) := \mid\beta (x_P)\mid=s.$$

Now, define the filtration $\bar\cf =\{ \bar\cf_m\}_{m\geq 0}$ of
$H^*(\fg ,\fl )\simeq H^*(X_P, \bc)$ by
  \[
\bar\cf_m := \text{Image of } H^*(\cf_m) \to H^*(C\u. ).
  \]
The filtration $\cf$ of $C\u.$ gives rise to the cohomology
spectral sequence $\{ E_r\}_{r\geq 1}$ converging to $H^*(C\u.
)=H^*(X_P,\bc )$.  By [K$_3$, Proof of Proposition 3.2.11], for any $m\geq
0$,
  \[
E^m_1 = \bigoplus_{s+t=m} [H\u._{(s)}(\fu_P)\otimes
H\u._{(t)}(\fu^-_P)]^{\fl},
  \]
where $H\u._{(s)}(\fu_P)$ denotes the cohomology of the subcomplex
$(\wed\u._{(s)}(\fu_P))^*$ of the standard cochain complex
$\wed\u. (\fu_P)^*$ associated to the Lie algebra $\fu_P$ and
similarly for $H\u._{(t)}(\fu^-_P)$.  Moreover, by loc. cit., the
spectral sequence degenerates at the $E_1$ term, i.e.,
  \begin{equation} \label{eqn10.1}
E_1^m = E^m_{\infty}.
  \end{equation}
Further, by the definition of $P$-relative height,
  \[
[H\u._{(s)}(\fu_P) \otimes H\u._{(t)}(\fu^-_P)]^{\fl} \neq 0 \Rightarrow
s=t.
  \]
Thus,
  \begin{align*}
E^m_1 &= 0, \qquad\quad\text{unless $m$ is even and}\\
E_1^{2m} &= [H\u._{(m)}(\fu_P) \otimes H\u._{(m)}(\fu^-_P)]^{\fl} .
  \end{align*}
In particular, from (\ref{eqn10.1}) and the general properties of spectral sequences
(cf. [K$_3$, Theorem E.9]), we have a canonical algebra isomorphism:
  \begin{equation} \label{eqn10.2}
\gr (\bar\cf ) \simeq \bigoplus_{m\geq 0} \bigl[ H\u._{(m)}(\fu_P) \otimes
H\u._{(m)}(\fu^-_P)\bigr]^{\fl} ,
  \end{equation}
where $\bigl[ H\u._{(m)}(\fu_P) \otimes H\u._{(m)}(\fu^-_P)\bigr]^{\fl}$
sits inside $\gr(\bar\cf )$ precisely as the homogeneous part of degree
$2m$; homogeneous parts of $\gr(\bar\cf )$ of odd degree being zero.

Finally, we claim that, for any $m\geq 0$,
  \begin{equation}  \label{eqn10.3}
\ca_m = \bar\cf_{2m} :
  \end{equation}

Following Kumar [K$_1$], take the d-$\partial$ harmonic representative
 $\hat{s}^w$ in $C\u.$ for the cohomology class $\epsilon_P^w$.  An
explicit expression is given in [K$_1$, Proposition 3.17].  From this explicit expression, we easily see that
  \begin{equation}  \label{eqn10.4}
\ca_m \subset \bar\cf_{2m}, \,\,\,\text{for all}\,\, m\geq 0.
  \end{equation}
Moreover, from the definition of $\ca$, for any $m\geq 0$,
  \[
\dim \frac{\ca_m}{\ca_{m+1}} = \# \bigl\{ w\in W^P : (\rho
-w^{-1}\rho) (x_P)=m\bigr\} .
  \]
Also, by the isomorphism (\ref{eqn10.2})  and [K$_3$, Theorem 3.2.7],
  $$
\dim \frac{\bar\cf_{2m}}{\bar\cf_{2m+1}} = \# \bigl\{ w\in W^P: (\rho-w^{-1}\rho)
(x_P) = m\bigr\}. $$

Thus,
  \beqn \label{eq77}
\dim \frac{\ca_m}{\ca_{m+1}} = \dim \frac{\bar\cf_{2m}}{\bar\cf_{2m+1}} .
  \eeqn
Of course, 
\beqn\label{eq78} \ca_0 = \bar\cf_0.
\eeqn
Thus,  combining the equations \eqref{eqn10.4}, \eqref{eq77} and \eqref{eq78}, we get \eqref{eqn10.3}. 
It is easy to see that the filtration $\{\bar{\cf}_{2m}\}_{m\geq 0}$ is multiplicative and hence so is
$\{\ca_m \}_{m\geq 0}$. This proves the (b) part of the proposition.
 \end{proof}

The  cohomology of $X_P$
 obtained by setting ${\tau}=0$ in
$(H^*(X_P, \Bbb Z)\otimes\Bbb{Z}[{\tau}],\odot)$ is denoted by
$(H^*(X_P, \Bbb Z),\odot_0)$. Thus, as a $\Bbb Z$-module, it  is the same as the singular cohomology
 $H^*(X_P, \Bbb Z)$ and under the product $\odot_0$ it is associative (and commutative).
\end{definition}

The following conjecture is a generalization of the corresponding result in the finite case due to Belkale-Kumar [BK, Theorem 22].
\begin{conjecture} \label{conj1}
Let $G$ be any indecomposable symmetrizable Kac-Moody  group (i.e., its generalized Cartan matrix is indecomposable, cf. [K$_3$, $\S$ 1.1]) 
and let  $(\lambda_1, \dots, \lambda_s, \mu)\in P_+^{s+1}$. Assume further that none of $\lambda_j$  is $W$-invariant and $\mu-\sum_{j=1}^s \lambda_j\in Q$, where $Q$ is the root lattice of $G$. 
Then, the following are equivalent:

(a)  $(\lambda_1, \dots, \lambda_s, \mu)\in \Gamma_s$.

(b) For every standard maximal parabolic subgroup $P$ in $G$ and every choice of
$s+1$-tuples  $(w_1, \dots, w_s, v)\in (W^P)^{s+1}$ such that $\epsilon_P^v$ occurs with coefficient $1$ in 
the deformed product
$$\epsilon_P^{w_1}\odot_0\, \cdots \,\odot_0 \epsilon_P^{w_s}
\in \bigl(H^*(X_P,\Bbb{Z}), \odot_0\bigr),$$
  the following inequality  holds:
    \[\label{eqn29}
\bigl(\sum_{j=1}^s \lambda_j(w_jx_{P})\bigr)-\mu(vx_P)\geq 0, \tag{$I^P_{(w_1,\dots, w_s,v)}$}
\]
where $\alpha_{i_P}$ is the (unique) simple root in $\Delta\setminus \Delta (P)$
and $x_P:=x_{i_P}$.
\end{conjecture}
\begin{remark}  (a) By Theorem \ref{thm1}, the above inequalities $I^P_{(w_1,\dots, w_s,v)}$ are indeed satisfied for any 
$(\lambda_1, \dots, \lambda_s, \mu)\in \Gamma_s$. 

(b) If $G$ is an affine Kac-Moody group, then the condition 
that $\lambda\in P_+$ is $W$-invariant is equivalent to the condition that $\lambda(c)=0$.
\end{remark}

\begin{theorem} \label{thm7.5} Let $\mathfrak{g}=\widehat{\mathfrak{sl}_{2}}$. 
Let $\lambda,\mu,\nu \in P_+$ be such that $\lambda+\mu-\nu\in Q$ and both of $\lambda(c)$ and $\mu(c)$ are nonzero. Then, the following are equivalent:

(a)   $(\lambda,\mu,\nu) \in \Gamma_2$.

(b) The following set of inequalities is satisfied for all $w\in W$ and $i=0,1$.
\begin{align*}
\lambda(x_i)+\mu(wx_i)-\nu(wx_i) &\geq 0, \,\,\,\text{and}\\
\lambda(wx_i)+\mu(x_i)-\nu(wx_i) &\geq 0.
\end{align*}
In particular, Conjecture \ref{conj1} is true for $\mathfrak{g}=\widehat{\mathfrak{sl}_{2}}$ and $s=2$.
\end{theorem}
\begin{proof}
By Lemma \ref{lemma5.1}, there exist (unique)  $n_1, n_2\in \bz$  such that 
\[
\nu - \mu + n_1 \delta \in P^o(\lambda), \,\,\,\,\text{and}\,\,\,
\nu - \lambda +n_2 \delta \in P^o(\mu).
\]
Let $n:=$ min $(n_1,n_2)$. 
By our description of the $\delta$-maximal components as in Theorem \ref{thm6.2} applied to $\bar{\lambda},\bar{\mu}, \bar{\nu}$
and using the identity \eqref{eq6.4.0}, we see that 
$L(\nu+n\delta)$ is a $\delta$-maximal component of $L(\lambda)\otimes L(\mu)$. Thus, by the equation \eqref{eq6.4.1}, for any $N\geq 1$, 
$L(N\nu+Nn\delta)$ is a $\delta$-maximal component of $L(N\lambda)\otimes L(N\mu)$. In particular, by Proposition \ref{tensor} and Lemma 
\ref{virasoro}, 
\beqn\label{eq7.4.1} L(N\nu) \subset L(N\lambda)\otimes L(N\mu) \,\,\,\text{ for some}\,\, N>1\,\,\,\text{ if and only if}\,\, n\geq 0.
\eeqn
By [Kac, Proposition 12.5 (a)], if a weight $\gamma+k\delta\in P(\lambda)$ (for some $k\in \bz_+$), then $\gamma\in P(\lambda)$. Thus, 
\beqn\label{eq7.4.2} n\geq 0 \,\,\,\text{ if and only if}\,\, \nu\in \bigl(P(\lambda)+\mu\bigr)\cap \bigl(P(\mu)+\lambda\bigr).
\eeqn
We next show that 
\beqn\label{eq7.4.3} P(\lambda)=(\lambda + Q)\cap C_\lambda,
\eeqn
where $C_\lambda:=\{\gamma\in \fh^*: \lambda(x_i)-\gamma(wx_i)\geq 0 \,\,\,\text{for all}\,\, w\in W \,\,\,\text{and all} \,\, x_i\}$.
Clearly, $$
 P(\lambda)\subset(\lambda + Q)\cap C_\lambda.$$
Since $\lambda+Q$ and $C_\lambda$ are $W$-stable, and $\lambda+Q$ is contained in the Tits cone (by [K$_3$, Exercise 13.1.E.8(a)]),
$(\lambda+Q)\cap C_\lambda= W\cdot \bigl((\lambda+Q)\cap C_\lambda\cap P_+\bigr)$.

Conversely,
take $\gamma\in  (\lambda + Q)\cap C_\lambda \cap P_+$. Then, 
$(\lambda-\gamma) (x_i)\geq 0$ and $(\lambda-\gamma) (c)= 0$ and hence $\lambda-\gamma \in \oplus_i\,\bz_+\alpha_i$, i.e., 
$\lambda \geq \gamma$. Thus, by [Kac, Proposition 12.5(a)], $\gamma \in P(\lambda)$. This proves \eqref{eq7.4.3}. 
Now, combining \eqref{eq7.4.1}, \eqref{eq7.4.2} and \eqref{eq7.4.3}, we get 
$ L(N\nu) \subset L(N\lambda)\otimes L(N\mu)$ for some $N>1$ if and only if for all 
 $w\in W$ and $i=0,1$,
\[
\lambda(x_i)-(\nu-\mu)(wx_i)\geq 0, \,\,\,\text{and}\,\,\,
\mu(x_i)-(\nu - \lambda)(wx_i) \geq 0.\]
This proves the equivalence of (a) and (b) in the theorem. 

To prove the `In particular' statement of the theorem, let $P_0$ (resp. $P_1$) be the maximal parabolic subgroup  of $G=\widehat{\SL_{2}}$
with $\Delta(P_0)=\{\alpha_1\}$ (resp.  $\Delta(P_1)=\{\alpha_0\}$). For any $n\geq 0$, let 
$$w_n:=\dots s_0s_1s_0 \,\,(\text{$n$-factors})\,\,\,\text{and}\,\,\,v_n:=\dots s_1s_0s_1 \,\,(\text{$n$-factors}).$$
Then, by [K$_3$, Exercise 11.3.E.4], $H^*(G/P_0)$ has a $\bz$-basis $\{\epsilon^{n}_{P_0}\}_{n\geq 0}$, where 
$\epsilon^{n}_{P_0}:=\epsilon^{w_n}_{P_0}$. Moreover,
$$\epsilon^{n}_{P_0}\cdot \epsilon^{m}_{P_0}=\binom {n+m}{n}\epsilon^{n+m}_{P_0}.$$
In particular, $\epsilon^{n+m}_{P_0}$ appears with coefficient one in $\epsilon^{n}_{P_0}\cdot \epsilon^{m}_{P_0}$ if and only if 
at least one of $n$ or $m$ is $0$. 

By using the diagram automorphism of $\widehat{\SL_{2}}$, one similarly gets that $H^*(G/P_1)$ has a $\bz$-basis $\{\epsilon^{n}_{P_1}\}_{n\geq 0}$, where 
$\epsilon^{n}_{P_1}:=\epsilon^{v_n}_{P_1}$, with the product given by
$$\epsilon^{n}_{P_1}\cdot \epsilon^{m}_{P_1}=\binom {n+m}{n}\epsilon^{n+m}_{P_1}.$$
Moreover, from the definition of the deformed product $\odot_0$,  clearly 
$$\epsilon^{0}_{P_0}\odot_0 \epsilon^{m}_{P_0}= \epsilon^{0}_{P_0}\cdot \epsilon^{m}_{P_0},$$
and similarly for $P_1$. From this the `In particular' statement of the theorem follows.
\end{proof}
\begin{remark} (1) It is easy to see that if $\lambda=m\delta$ for some $m\in \bz$, then the equivalence of (a) and (b) in the above theorem breaks down.

(2) Though we have proved Conjecture \ref{conj1} for $\widehat{\SL_{2}}$ only for $s=2$, it is quite likely that a similar proof will prove it for 
any $s$ (for $\widehat{\SL_{2}}$).
\end{remark}

\section{The $A_2^{(2)}$ case}
By a method similar to that of $A_1^{(1)}$,  we  handle the  $A_2^{(2)}$ case, with minor modifications where necessary. 
Write $\mathfrak{h}=\mathbb{C}c\oplus\mathbb{C}\alpha^{\vee}\oplus\mathbb{C}d$
and $\mathfrak{h}^{*}=\mathbb{C}\omega_{0}\oplus\mathbb{C}\alpha\oplus\mathbb{C}\delta$,
where $\alpha(\alpha^{\vee})=2$, $\delta(d)=1$, $\omega_{0}(c)=1$,
and all other values are $0$.  Then $(\mathfrak{h},\{\alpha_0:=\delta-2\alpha, \alpha_1:=\alpha\}, \{\alpha_0^\vee:=c-\frac{1}{2}\alpha^{\vee},\alpha_1^\vee:=\alpha^{\vee}\})$ is a realization of the GCM 
\[ \left( \begin{array}{cc}
2 & -1\\
-4 & 2  \end{array} \right)\] 
of $A_2^{(2)}$.
The fundamental weights are $\omega_{0}$ and $\omega_{1}=\frac{1}{2}\omega_{0}+\frac{1}{2}\alpha$. This easily allows one to compute the dominant $\delta$-maximal weights. Analogous to Corollary \ref{cor5.1}, we have the following:
\begin{lemma}\label{domdelmaxweights}  Let $\lambda$ be a dominant integral weight. Then, the dominant $\delta$-maximal
weights of $L(\lambda)$ are the dominant weights of the form 
\[
P_+\cap \left\{ \lambda-j\alpha,\,\lambda+k(2\alpha-\delta),\,\lambda+\alpha-\delta+l(2\alpha-\delta)\,:\, j,k,l\in\mathbb{Z}_{\geq0}\right\} .
\]
Moreover,  $P^o(\lambda)$ is the $W$-orbit of the above.
\end{lemma}

Again, to determine the saturated tensor cone, it is enough to describe the $\delta$-maximal components.  Thus, to determine the $\delta$-maximal components, by virtue of proposition \ref{tensor}, we must find the highest $\delta$-degree term in $\sum_{\lambda\in T_{\Lambda}^{\Lambda',\Lambda''}}\varepsilon(v_{\Lambda,\Lambda'',\lambda})c_{\Lambda',\lambda}e^{S_{\Lambda,\Lambda'',\lambda}\delta}$. This computation is done in a somewhat similar manner as in the $A_1^{(1)}$ case, but there are some important modifications.
First, we need to use two different piecewise smooth functions to describe the $\delta$-maximal weights of $L(\lambda)$. An upper function $A^+$  interpolates the $\delta$-maximal weights which are in the $W$-orbit of the dominant weights of the form
\[
\left\{ \lambda-j\alpha,\,\lambda+k(2\alpha-\delta)\,:\, j,k\in\mathbb{Z}_{\geq0}\right\}, 
\]
while another function $A^-$ interpolates the $\delta$-maximal weights in the $W$-orbit of the dominant weights of the form
\[
\left\{ \lambda-j\alpha,\,\lambda+\alpha-\delta+l(2\alpha-\delta)\,:\, j,l\in\mathbb{Z}_{\geq0}\right\}. 
\]

Now, all of the arguments made in the $\widehat{\mathfrak{sl}_2}$ case must be made for two extensions of $S_{\Lambda,\Lambda'',\lambda}$ to non-integral values, using $A^+$  and $A^-$ respectively. Let $\Lambda :=m_{0}\omega_{0}+m_{1}\omega_{1}$, $\Lambda' :=m'_{0}\omega_{0}+m'_{1}\omega_{1}$, and $\Lambda'' :=m''_{0}\omega_{0}+m''_{1}\omega_{1}$. The following result is an analogue of 
Proposition \ref{maximum} and Lemma \ref{lemma5.9} for the $A_2^{(2)}$ case.
\begin{prop}\label{maxS} Let $\Lambda, \Lambda', \Lambda''$ be as above. Assume that both of $\Lambda'(c)$ and $\Lambda''(c)>0$ and $\Lambda-\Lambda'-\Lambda''\in Q$, where 
$Q=\bz \alpha+\bz \delta$ is the root lattice of $A_2^{(2)}$ . 
Then, the  maximum $\mu^{\Lambda',\Lambda''}_\Lambda$ of the set
\[ 
\left\{ S_{\Lambda,\Lambda'',\lambda}:\;\lambda\in T_{\Lambda}^{\Lambda',\Lambda''},\;\varepsilon(v_{\Lambda,\Lambda'',\lambda})=1\right\}
\] occurs when $\lambda\equiv\Lambda'+\frac{1}{2}\left(m_{1}-m_{1}'-m_{1}''\right)\alpha \mod\mathbb{C}\delta$.
The maximum $\bar{\mu}^{\Lambda',\Lambda''}_\Lambda$ of the set 
\[
\left\{ S_{\Lambda,\Lambda'',\lambda}:\;\lambda\in T_{\Lambda}^{\Lambda',\Lambda''},\;\varepsilon(v_{\Lambda,\Lambda'',\lambda})=-1\right\}
\]
occurs when $\lambda\equiv\Lambda'-\bigl(\frac{1}{2}(m_{1}'+m_{1}''+m_{1})+1\bigr)\alpha \mod \mathbb{C}\delta $
or when $\lambda\equiv\Lambda'-\bigl(\frac{1}{2}(m_{1}'+m_{1}''+m_{1})-2(\Lambda'(c)+\Lambda''(c)+1)\bigr)\alpha  \mod \mathbb{C}\delta $. 
\end{prop}

\begin{corollary}\label{a22cancellation}
Let $\Lambda, \Lambda', \Lambda''$ be as in Proposition \ref{maxS}. Assume further that 
$\Lambda'(c) \geq 2$, $\Lambda''(c) \geq 2$,  $m'_1,m''_1 \neq 1$. Then,  if ${\mu}^{\Lambda',\Lambda''}_\Lambda = \bar{\mu}^{\Lambda',\Lambda''}_\Lambda$, we have 
\[ {\mu}^{\Lambda'',\Lambda'}_\Lambda \neq \bar{\mu}^{\Lambda'',\Lambda'}_\Lambda.\]
\end{corollary}

The proof of Corollary \ref{a22cancellation} requires a description of the situations in which 
 ${\mu}^{\Lambda',\Lambda''}_\Lambda = \bar{\mu}^{\Lambda',\Lambda''}_\Lambda$. We reduce these situations to certain cases, and show that in most of these cases, if the roles of $\Lambda'$ and $\Lambda''$ are interchanged, then (as in the $\widehat{\mathfrak{sl}_2}$ case) the equality does not occur. In the remaining cases, we show that  $\Lambda'(c) < 2$, $\Lambda''(c) < 2$,  $m'_1 =1$,  or $m''_1 = 1$.

\begin{thm} Let $\Lambda,\Lambda',\Lambda''$ be as in Proposition \ref{maxS}.  Then, 
$L(\Lambda+n\delta)$ is a $\delta$-maximal component of $L(\Lambda')\otimes L(\Lambda'')$
if $n=\min(n_{1},n_{2})$, where $n_{1}$ is such that $\Lambda-\Lambda''+n_{1}\delta\in P^o(\Lambda')$ and
$n_{2}$ is such that $\Lambda-\Lambda'+n_{2}\delta\in P^o(\Lambda'')$.
\end{thm}
\begin{lem} Fix a positive integer $N$. 
Let $\Lambda\in \bar{P}_+$ and let $\lambda\in\Lambda+Q$. 
Then, $N\lambda\in P^o(N\Lambda)$ if
and only if $\lambda\in P^o(\Lambda)$. 
\end{lem}
Combining the above results, we get a description of $\Gamma_2$, which is identical to that of $\widehat{\mathfrak{sl}_2}$ 
(cf. Theorem \ref{thm7.5}).
\begin{theorem} \label{thm8.6} Let $\mathfrak{g}=A_2^{(2)}$. 
Let $\lambda,\mu,\nu \in P_+$ be such that $\lambda+\mu-\nu\in Q$ and both of $\lambda(c)$ and $\mu(c)$ are nonzero. Then, the following are equivalent:

(a)   $(\lambda,\mu,\nu) \in \Gamma_2$.

(b) The following set of inequalities is satisfied for all $w\in W$ and $i=0,1$.
\begin{align*}
\lambda(x_i)+\mu(wx_i)-\nu(wx_i) &\geq 0, \,\,\,\text{and}\\
\lambda(wx_i)+\mu(x_i)-\nu(wx_i) &\geq 0.
\end{align*}
In particular, Conjecture \ref{conj1} is true for this case as well for $s=2$. 
\end{theorem}
 The `In particular' statement of the above theorem follows by using the description of the cup product in the cohomology of the full flag variety of $A_2^{(2)}$ given by Kitchloo [Ki]. 

It is clear that if the level of $L(\Lambda')$ or $L(\Lambda'')$ is zero, then the tensor product has a single component. Thus, it is  already saturated. Assume now that the levels of both of $L(\Lambda')$ and $L(\Lambda'')$ are $>0$. Then, since there are representations of level $\frac{1}{2}$, the conditions of Corollary \ref{a22cancellation} are satisfied for any $N\Lambda$, $N\Lambda'$, $N\Lambda''$ with  $\Lambda-\Lambda'-\Lambda''\in Q$,   provided $N \geq 4$.  Hence:
\begin{cor} \label{cor8.7} For $A_2^{(2)}$, $4$ is a saturation factor.
\end{cor}
\begin{remark} 
When the Kac-Moody Lie algebra $\fg$ is infinite dimensional, then the saturated tensor semigroup $\Gamma_s$ is {\em not}  finitely generated,
for any $s\geq 2$. Thus, it is not clear a priori that there  exists a saturation factor for such a $\fg$.
\end{remark}

\vskip5ex

\noindent
Address: Department of Mathematics,
University of North Carolina,
Chapel Hill, NC  27599--3250 \\
\noindent
(emails: shrawan@email.unc.edu; merrickb@email.unc.edu).
\end{document}